\documentclass[11pt]{amsart}

\usepackage{amsfonts,color}
\usepackage{latexsym,amssymb}
\usepackage{amsmath}
\usepackage[utf8]{inputenc}

\newtheorem{theorem}{Theorem}[section]
\newtheorem{lemma}[theorem]{Lemma}
\newtheorem{proposition}[theorem]{Proposition}

\newtheorem{remark}[theorem]{Remark}

\newcommand{\dx}{\,\mathrm{d}x}
\newcommand{\dy}{\,\mathrm{d}y}

\usepackage[bookmarksnumbered,colorlinks, linkcolor=blue, citecolor=red, pagebackref, bookmarks, breaklinks]{hyperref}

\usepackage[hyperpageref]{backref}

\catcode`@=11 \@addtoreset{equation}{section} \catcode`@=12

\usepackage[left=3cm,right=3cm,top=3cm,bottom=3cm]{geometry}

\begin{document}
	
	\title[Hartree-Fock systems with null mass with logarithmic interactions]{A class of Hartree-Fock systems with null mass via  Nehari-Pohozaev with logarithmic interactions}

	\author[J. C. de Albuquerque]{J. C. de Albuquerque}
	\thanks{The first author was partially supported by CNPq with grants 301849/2025-0.}
	\address{Departamento de Matem\'atica, Universidade Federal de Pernambuco \\
		50670-901, Recife--PE, Brazil}
	\email{josecarlos.melojunior@ufpe.br}
	
	\author[J. Carvalho]{J. Carvalho}
	\thanks{The second author was partially supported by CNPq with grant 300853/2025-4.}
	
	\address{Departamento de Matemática, Universidade Federal de Sergipe\\
		49100-000, São-Cristóvão--SE, Brazil}
	\email{jonison@mat.ufs.br}

	\author[E. D. Silva]{E. D. Silva*}
	\thanks{The third author was also partially supported by CNPq with grant 310440/2025-4.}
	\address{Instituto de Matemática e Estatística,
		Universidade Federal de Goiás
		74001-970, Goiânia-GO, Brazil}
	\email{edcarlos@ufg.br}
	\thanks{Corresponding author: E. D. Silva email: edcarlos@ufg.br}
	
	
	
	\begin{abstract} We establish the existence and qualitative properties of nontrivial solutions for a class of Hartree-Fock type systems defined over the whole space $\mathbb{R}^2$. By introducing a suitable Nehari-Pohozaev manifold, we prove the existence, regularity and we describe the asymptotic behavior of solutions with respect to the interaction parameter $\beta > 0$. In particular, we show that the system admits either a vector ground state or a semitrivial ground state solution, depending on the magnitude of $\beta$.	
	\end{abstract}

	\maketitle
	
\tableofcontents
	
	\section{Introduction and main results}
	
	The Hartree-Fock type system has attracted considerable attention in theoretical research over the years. This class of systems appears in the study of standing wave solutions for an approximation of the Hartree model with a two-component attractive interaction, and its form is given by
	\begin{equation}\label{sistema-fisico}
		\left\{
		\begin{aligned}
			-i \varphi_t-\Delta_x \varphi+\varphi+  \phi\varphi&=|\varphi|^{2p-2}\varphi+\beta |\psi|^{p}|\varphi|^{p-2}\varphi , & \mbox{in} \ \mathbb{R}^N\times(0,\infty),
			\\
			-i \psi_t -\Delta_x \psi + \psi +  \phi \psi&=|\psi|^{2p-2}\psi+\beta|\varphi|^p|\psi|^{p-2}\psi, & \mbox{in} \ \mathbb{R}^N\times(0,\infty),
		\end{aligned}
		\right.
	\end{equation}
	with $N\geq2$, $\beta \geq 0$, and $1<p<\infty$. The pair $(\varphi,\psi)$, with $\varphi,\psi:\mathbb{R}^N \to \mathbb{C}$, represent a standing wave solution. The System \eqref{sistema-fisico} is motivated by an approximation given in Hartree-Fock method which is considered in orbital theory (see \cite{fock,hartee,slater}). See also \cite{beru,catto,lieb2,mar}. 
	
	Inspired by \eqref{sistema-fisico}, we are interested in standing wave solutions of the form $(\varphi,\psi ) = (e^{it}u, e^{it}v)$, with $u,v:  \mathbb{R}^N \to \mathbb{R}$. In this case, $(u,v)$ is a solution of 
	\begin{equation}\label{sistema-massa-positiva}
		\left\{
		\begin{aligned}
			-\Delta u + V(x)u + \phi u& =  |u|^{2p-2}u + \beta |v|^{p} |u|^{p-2}u, & \mbox{in} \ \mathbb{R}^N,
			\\
			-\Delta v + V(x)v + \phi v& = |v|^{2p-2}v + \beta |u|^{p}|v|^{p-2}v, & \mbox{in} \ \mathbb{R}^N,
		\end{aligned}
		\right.
	\end{equation}
    with $V \equiv1$. The term $\phi$ is obtained by the convolution of $\gamma\left(u^2+v^2\right)$ (for some $\gamma \in \mathbb{R}$) with the fundamental solution of the Laplacian $\Gamma:\mathbb{R}^N \setminus\{0\} \to \mathbb{R}$ given by
	$$
	\Gamma(x)=\frac{1}{2\pi}\log|x|, \ \ \mbox{ if} \ \ N=2 \qquad \mbox{and} \qquad
	\Gamma(x)=\frac{c_N}{|x|^{N-2}}, \ \ \mbox{ if} \ \ N \geq 3,
	$$
	where $c_N=1/N(N-2)\omega_{N}$ ($\omega_{N}$ is the volume of the unit in $\mathbb{R}^N$). More precisely,
	$$
	\phi(x)=\phi_{u,v}(x):=\left[\Gamma \ast \left(\gamma\left(u^2+v^2\right)\right)\right](x)=\gamma\int_{\mathbb{R}^N}\Gamma(|x-y|)\left(u^2(y)+v^2(y)\right)\dy.
	$$
	In this case, the function $\phi_{u,v}$ (Newtonian potential) is the natural candidate for solving the following Poisson's equation
	\begin{equation}\label{eq-Poisson}
		-\Delta\phi=\gamma\left(u^2+v^2\right), \quad \mbox{in} \ \mathbb{R}^N.
	\end{equation}

	The system \eqref{sistema-massa-positiva}, with $N\geq 3$, has been intensively studied in the last years, see \cite{hf1,lieb3,lions,Zhu-Zushun}. Considering  $v=0$ in \eqref{sistema-massa-positiva}, notice that this system has a slight relation with the  equation 
	\begin{equation}\label{caso-escalar}
		-\Delta u+V(x)u+\phi_{u,0} u=|u|^{2p-2}u, \quad  \mbox{in} \ \mathbb{R}^N.
	\end{equation}
	In dimension three, variations of this problem have been studied in \cite{hf1,benci1, benci2,CerVai,Ruiz}, and many others. The planar case is more delicate due to the sign-changing nature of the $\log|x|$ kernel, in contrast to the strictly positive behavior of $1/|x|^{N-2}$ in higher dimensions. Moreover, when $N\geq 3$, we can  use the Lax-Milgram Theorem in the space $D^{1,2}(\mathbb{R}^N)$ to  solve the equation \eqref{eq-Poisson}. This kind of approach can not be applied for the case $N =2$ due to the loss of continuous embedding of $D^{1,2}(\mathbb{R}^2)$ into the Lebesgue spaces. In this direction, we can cite \cite{Cingolani-Weth,Stubbe,Cingolani-Jeanjean}.
	
	Depending on the behavior of $V$, for example, such as being bounded away from zero and sufficiently large in an appropriate sense, or being periodic, one can establish existence and multiplicity of solutions for problems like \eqref{caso-escalar}, using, for instance, the Mountain Pass Theorem or the Nehari method. On this subject we refer the reader to \cite{alves,ambro,azzo,bon,cv2023,Cingolani-Weth}.

	The case $V\equiv0$ where is called \textit{problem with zero mass (or null mass)}. For the case $N \geq 3$, we can cite \cite{claudianor-jianfu,Gu-Liao,Zhang-Dongdong-Siti-Wu}.
    
    In the planar case, we could
	cite \cite{Chen-Shu-Tang-Wen,Wen-Chen-Radulescu,Yang,Wei-Tang-Zhang}. In these works the authors use functions with symmetry, called \textit{axially symmetric functions} ($u(x,y)=u(|x|,|y|)$). For further results for semilinear elliptic problems considering the zero mass case we also refer the reader to \cite{nodea2025,du,tang,wang}.

	Our goal in this work is to study the system \eqref{sistema-massa-positiva} in dimension two, focusing on the case $V \equiv 0$. The case $V \equiv 1$ was previously investigated in \cite{Carvalho-Figueiredo-Maia-Medeiros}, under the restriction to radially symmetric functions. A key contribution of our work is that we do not impose any symmetry assumptions, such as axial or radial symmetry, on the function space. More precisely, in the present work, we consider the following class of planar elliptic systems:
	\begin{equation}\tag{$\mathcal{S}_{\beta}$}\label{S}
		\left\{
		\begin{aligned}
			-\Delta u+\phi_{u,v} u&=|u|^{2p-2}u+\beta|v|^p|u|^{p-2}u, & \mbox{in} \ \mathbb{R}^2,
			\\
			-\Delta v+\phi_{u,v} v&=|v|^{2p-2}v+\beta|u|^p|v|^{p-2}v, & \mbox{in} \ \mathbb{R}^2,
		\end{aligned}
		\right.
	\end{equation}	
	where $2\leq p<\infty$ and $\beta\geq0$. By taking $\gamma=2\pi$, we have that
	\begin{equation}\label{potential-new}
		\phi_{u,v}(x):=\int_{\mathbb{R}^2}\log(|x-y|)\left(u^2(y)+v^2(y)\right)\dy.
	\end{equation}
	Due to the presence of the term $\phi_{u,v}$, inspired by \cite{Stubbe}, the appropriate workspace for the system given above can be written in the following form: 
	$$
	W^\lambda:=\left\lbrace u\in H^1(\mathbb{R}^2) : \|u\|_{\lambda,*}^2:=\int_{\mathbb{R}^2}\log(\lambda+|x|)u^2\dx<\infty \right\rbrace,
	$$
	where the norm is defined by 
	$$
	\|u\|_{W^\lambda}:=\left(\|\nabla u\|_{2}^2 + \|u\|_{2}^2+\|u\|_{\lambda,*}^2\right)^{1/2},
	$$
	for $\lambda>0$. The energy functional $I_\beta: W^\lambda \times W^\lambda \to \mathbb{R}$ associated to this system is given by
	$$
	I_{\beta}(u,v)=\frac{1}{2}\left(\|\nabla u\|_{2}^2+\|\nabla v\|_{2}^2\right) + \frac{1}{4}\mathcal{V}(u,v)-\frac{1}{2p}\psi_\beta(u,v),
	$$
	where
	$$
	\mathcal{V}(u,v):=\int_{\mathbb{R}^2}\int_{\mathbb{R}^2}\log(|x-y|)\left(u^2(y)+v^2(y)\right)\left(u^2(x)+v^2(x)\right)\dy\dx
	$$
	and
	\begin{equation}\label{psi}
		\psi_\beta(u,v):=\|u\|_{2p}^{2p}+\|v\|_{2p}^{2p}+2\beta\int_{\mathbb{R}^2}|uv|^p\dx.
	\end{equation}
	Here, we denote the norm of Lebesgue space $L^s(\mathbb{R}^2)$ by $\|\cdot\|_s$, with $1\leq s<\infty$. 
	
	In order to obtain least energy solution it is common to use the Nehari manifold
	$$
	\mathcal{N}_\beta:=\left\lbrace  (u,v) \in \left(W^\lambda \times W^\lambda\right) \setminus \{(0,0)\} : I_\beta'(u,v)(u,v)=0 \right\rbrace,
	$$
	where
	\begin{equation}\label{derivada-Nehari}
		I_\beta'(u,v)(u,v)=\|\nabla u\|_{2}^2+\|\nabla v\|_{2}^2+\mathcal{V}(u,v)-\psi_\beta(u,v).
	\end{equation}
	In general, the fiber map $t \mapsto I_\beta(t u,tv)$ is used in this context. However, this strategy does not work in the case of zero mass, since the term $\|\cdot\|_2$ does not appear in the energy functional, which is part of the norm in the space $H^1(\mathbb{R}^2)$. Here, we need to apply some minimization arguments for a Nehari-Pohozaev manifold taking into account the fibers $t \mapsto I_\beta(t^2u_t,t^2v_t)$, where $u_t(x):=u(tx)$ and $v_t(x):=v(tx)$. It can be obtained in a non-rigorous way by the following relation
	$$
	\frac{\mathrm{d}}{\mathrm{d}t}\left(I_{\beta}\left(t^2u_t,t^2v_t\right)\right)\Big|_{t=1} = 0.
	$$
	More precisely, the Nehari-Pohozev manifold is defined by
	$$
	\mathcal{M}_{\beta}:=\left\lbrace (u,v)\in \left(W^\lambda\times W^\lambda\right)\setminus\{(0,0)\} : J_{\beta}(u,v)=0 \right\rbrace, 
	$$
	where
	\begin{equation}\label{Juv}
		\begin{aligned}
			J_{\beta}(u,v)&:=2\left(\|\nabla u\|_{2}^2+\|\nabla v\|_{2}^2\right) - \frac{1}{4}\left(\|u\|_{2}^2+ \|v\|_{2}^2\right)^2
			+ \mathcal{V}(u,v) - \frac{(4p-2)}{2p}\psi_\beta(u,v).
		\end{aligned}
	\end{equation}
	Unlike problems involving the standard Nehari manifold $\mathcal{N}_\beta$, it is not immediate that all critical points of $I_\beta$ lie on $\mathcal{M}_\beta$. This makes the application of the Nehari-Pohozaev approach more challenging when aiming to guarantee the existence of a least energy solution. To overcome this difficulty, we will first need to rigorously establish the following Pohozaev identity
	$$
	\mathcal{P}_\beta(u,v):=\mathcal{V}(u,v)+\frac{1}{4}\left(\|u\|_2^2+\|v\|_2^2\right)^2-\frac{1}{p}\psi_\beta(u,v)=0,
	$$
	and then verify that
    \begin{equation*}
        J_\beta(u,v)=2I_\beta'(u,v)(u,v)-\mathcal{P}_\beta(u,v).
    \end{equation*}

	A pair $(u,v) \in W^\lambda\times W^\lambda$ is said to be a weak solution for System \eqref{S}, whenever 
	$$
	\int_{\mathbb{R}^2} \nabla u\nabla\varphi \dx +  \int_{\mathbb{R}^2}\phi_{u,v}u\varphi\dx = \int_{\mathbb{R}^2}\left(|u|^{2p-2}u+\beta|v|^p|u|^{p-2}u\right)\varphi\dx
	$$
	and
	$$
	\int_{\mathbb{R}^2} \nabla v\nabla\psi \dx +  \int_{\mathbb{R}^2}\phi_{u,v}v\psi\dx = \int_{\mathbb{R}^2}\left(|v|^{2p-2}v+\beta|u|^p|v|^{p-2}v\right)\psi\dx
	$$
	for any $\varphi,\psi\in C_0^\infty(\mathbb{R}^2)$. Under these conditions, it follows that $(u,v) \in W^\lambda \times W^\lambda$ is a critical point for $I_\beta$ if, and only if, $(u,v)$ is a weak solution for the System \eqref{S}. We refer to a solution $(u,v)$ as a semi-trivial solution when either $u = 0$ or $v = 0$, and as a vectorial solution when both components are nontrivial, that is, $u,v \neq 0$. A solution $(u,v)$ is said to be nonnegative if $u, v \geq 0$. Furthermore, a pair $(u,v) \in W^\lambda \times W^\lambda$ is called a least energy solution (or ground state solution) to System \eqref{S} if $I_\beta(u,v) \leq I_\beta(z, w)$ for all nontrivial weak solutions $(z, w) \in W^\lambda \times W^\lambda$.
	
		Our first main result can be stated as follows:
		
		\begin{theorem}[Existence and asymptotic behavior of solutions]\label{existencia-de-solucao}
			Assume that $2\leq p<\infty$ and $\beta\geq0$. Then, System \eqref{S} admits a nonnegative weak solution $(u_\beta^\lambda,v_\beta^\lambda) \in  \mathcal{M}_\beta$ such that
			$$
			I_\beta(u_\beta^\lambda,v_\beta^\lambda)=c_{\beta}:=\inf_{(u,v)\in\mathcal{M}_{\beta}}I_{\beta}(u,v),
			$$
			for each  $\lambda > e^{1/4}$.
		Furthermore,
		\begin{itemize}
			
	\item [$(a)$] for $\lambda > e^{1/4}$ fixed, $u_\beta^\lambda\rightharpoonup u^\lambda$ and $v_\beta^\lambda\rightharpoonup v^\lambda$ weakly in $W^\lambda$, as $\beta\to0^+$, such that
	$$
	I_0(u^\lambda,v^\lambda)=c_0;
	$$
	
	\item[$(b)$] for $\lambda > e^{1/4}$ fixed, we obtain that  $\|u_\beta^\lambda\|_2 \to 0$ and $\|v_\beta^\lambda\|_2 \to 0$, as $\beta\to\infty$. 
	\end{itemize}
\end{theorem}

Now, we shall discuss some results on regularity of solutions for system \eqref{S}. Furthermore, we prove a Pohozaev identity which permits to show that $\mathcal{M}_\beta$ is a natural constraint. Namely, we can state the following result: 
	
	\begin{theorem}[Regularity and Pohozaev identity]\label{teo-Pohozaev}  Assume the conditions of Theorem~\ref{existencia-de-solucao} and let $(u,v)$ be a solution of System \eqref{S}. Then,
	
	\begin{itemize}
		\item [$(i)$] $u,v \in C^{2,\sigma}_{\mathrm{loc}}(\mathbb{R}^2)$, for some $\sigma\in(0,1)$;
		
		\item[$(ii)$] if $u \neq 0$, then $u>0$. Similarly, if $v\neq0$, we have $v>0$;
		
		\item[$(iii)$] it holds that
		\begin{equation}\label{identidade-Pohozaev}
			\mathcal{P}_\beta(u,v):=\mathcal{V}(u,v)+\frac{1}{4}\left(\|u\|_2^2+\|v\|_2^2\right)^2-\frac{1}{p}\psi_\beta(u,v)=0.
		\end{equation}
	\end{itemize}	
	\end{theorem}

	\begin{remark}[Ground state solution]
		From \eqref{derivada-Nehari}, \eqref{Juv}, and \eqref{identidade-Pohozaev} we can conclude that
		$$
		J_\beta(u,v)=2I_\beta'(u,v)(u,v)-\mathcal{P}_\beta(u,v).
		$$
		In view of Theorem~\ref{teo-Pohozaev}, if $(z,w)$ is a weak solution of System \eqref{S}, then $(z,w) \in \mathcal{M}_\beta$. Therefore, the solution   that achieve the level $c_\beta$ is a ground state solution.
	\end{remark}

	Now, we are able to consider the existence of semitrivial solutions of System \eqref{S} which are given by the vectors $(u,0)$ or $(0,v)$, where $u, v$ are nontrivial weak solution to the following scalar problems, respectively: 
		\begin{equation}\tag{$\mathcal{P}_{1}$}\label{P1}
		\begin{aligned}
			-\Delta u+\phi_{u} u&=|u|^{2p-2}u, & \mbox{in} \ \mathbb{R}^2,
		\end{aligned}
	\end{equation}	
	and 
		\begin{equation}\tag{$\mathcal{P}_{2}$}\label{P2}
		\begin{aligned}
			-\Delta v+\phi_{v} v&=|v|^{2p-2}v, & \mbox{in} \ \mathbb{R}^2.
		\end{aligned}
\end{equation}	
Recall also that $\phi_u = \phi_{u,0}$ and $\phi_v = \phi_{0,v}$, see \eqref{potential-new}.

Another feature in the present work is to find sharp conditions on the parameter $\beta > 0$ such that System \eqref{S} admits semi-trivial solutions or vector solutions. Namely, we shall state the following result: 

\begin{theorem}[Vectorial and semi-trivial solutions]\label{vetorial-semitrivial}
	Considering the pair $(u_\beta^\lambda,v_\beta^\lambda)=(u_\beta,v_\beta)$ (with $\lambda>e^{1/4}$ fixed) obtained in Theorem~\ref{existencia-de-solucao}, it follows that:
	\begin{itemize}
		\item [$(i)$] for any $\beta>2^{p-1}-1$, the pair $(u_\beta,v_\beta)$ is a vector solution;
		
		\item [$(ii)$] for any $0\leq\beta<2^{p-1}-1$, the pair $(u_\beta,v_\beta)$ is a semi-trivial solution;
		
		\item[$(iii)$] $\beta=2^{p-1}-1$, if and only if, $(u/\sqrt{2},u/\sqrt{2})$ is a solution of \eqref{S}, where $u$ is a solution of the scalar equation given in the problem \eqref{P1}.
	\end{itemize}
\end{theorem}

\begin{remark}
    It is worth emphasizing that a major difficulty in dealing with the class of systems \eqref{S} lies in the fact that we could not establish that $\mathcal{M}_\beta$ constitutes a natural constraint for $I_\beta$ through the Lagrange Multiplier Theorem. To overcome this obstacle, we employed the Quantitative Deformation Lemma (see Proposition \ref{esquema}).
\end{remark}

\begin{remark}
    In \cite{cv2023,nodea2025}, the authors investigated the $(2,q)$-Laplacian operator with $1<q<2$. Our contribution complements these results by addressing the remaining case $q=2$, namely the classical Laplacian.
\end{remark}

\textbf{Outline.} In the forthcoming section, we introduce the variational structure and the Nehari manifold. Section \ref{Sec-N-P-M} is devoted to introduce a Nehari-Pohozaev manifold and your proprieties. In the Section \ref{Sec-Teo-existencia}, we prove of Theorem \ref{existencia-de-solucao}. The Section \ref{sec-regularidade-Pohozaev} is dedicated to the proof of regularity and Pohozaev identity. Finally, in Section \ref{semitrivial-vetorial}, we present the proof to establish semi-trivial and vector solutions.

	\section{Variational framework}

        In this Section we establish the variational framework to our problem. Recall that the energy functional $I_\beta: W^\lambda \times W^\lambda \to \mathbb{R}$ associated to System \eqref{S} is given by
	$$
	I_{\beta}(u,v)=\frac{1}{2}\left(\|\nabla u\|_{2}^2+\|\nabla v\|_{2}^2\right) + \frac{1}{4}\mathcal{V}(u,v)-\frac{1}{2p}\psi_\beta(u,v),
	$$
	where
	$$
	\mathcal{V}(u,v):=\int_{\mathbb{R}^2}\int_{\mathbb{R}^2}\log(|x-y|)\left(u^2(y)+v^2(y)\right)\left(u^2(x)+v^2(x)\right)\dy\dx
	$$
	and
	\begin{equation*}
		\psi_\beta(u,v):=\|u\|_{2p}^{2p}+\|v\|_{2p}^{2p}+2\beta\int_{\mathbb{R}^2}|uv|^p\dx.
	\end{equation*}
        By using the identity $\log r=\log(\lambda+r)-\log(1+\lambda r^{-1})$, for any $r,\lambda>0$, we can write 
	\begin{equation}\label{decomposicao}
		\mathcal{V}(u,v)=\mathcal{V}_1(u,v)-\mathcal{V}_2(u,v),
	\end{equation}
	where
	$$
	\mathcal{V}_1(u,v):=\int_{\mathbb{R}^2}\int_{\mathbb{R}^2}\log\left(\lambda+|x-y|\right)\left(u^2(y)+v^2(y)\right)\left(u^2(x)+v^2(x)\right)\dy\dx,
	$$
	and
	$$
	\mathcal{V}_2(u,v):=\int_{\mathbb{R}^2}\int_{\mathbb{R}^2}\log\left(1+\lambda|x-y|^{-1}\right)\left(u^2(y)+v^2(y)\right)\left(u^2(x)+v^2(x)\right)\dy\dx.
	$$
	As proved in \cite[Lemma 2.2]{Cingolani-Weth}, the non-local terms $\mathcal{V}_1,\,\mathcal{V}_2$ are well-defined and  belong to $C^1(W^\lambda\times W^\lambda)$. Furthermore, taking into account that  $\lambda+|x-y|\leq(\lambda+|x|)(\lambda+|y|)$, for any $x,y\in\mathbb{R}^{2}$ and $\lambda\geq1$, we have that
	\begin{equation}\label{des-log-elementar}
		\log(\lambda+|x-y|)\leq\log((\lambda+|x|)(\lambda+|y|))=\log(\lambda+|x|)+\log(\lambda+|y|).
	\end{equation}
	By using a straightforward computation, we estimate
	\begin{equation}\label{desigualdadeV1}
		\mathcal{V}_1(u,v)\leq 2\left(\|u\|^2_2+\|v\|^2_2\right)\left(\|u\|^2_{W^\lambda}+\|v\|^2_{W^\lambda}\right),
	\end{equation}	
	for any $u,v \in W^\lambda$.
	In order to estimate $\mathcal{V}_2$, we will use the Hardy-Littlewood-Sobolev inequality, see \cite{lieb1}, which can be stated as follows:
	\begin{proposition}\label{hls} Let $q,\,s>1$, $0<\mu<2$ with $1/q+1/s+\mu/2=2$, $g\in L^q(\mathbb{R}^2)$ and $h\in L^s(\mathbb{R}^2)$. Then, there exists a constant  $C(q,s,\mu)>0$ such that
		$$
		\left| \int_{\mathbb{R}^2}\int_{\mathbb{R}^2}\frac{g(y)h(x)}{|x-y|^\mu}\dy\dx\right|\leq C(q,s,\mu)\|g\|_q\|h\|_s.
		$$
	\end{proposition}
	Applying Proposition \ref{hls} with $\mu=1$, $q=s=4/3$, and using the elementary inequality $\log(1+r)\leq r$, for any $r>0$ together with the Sobolev embedding $W^\lambda\hookrightarrow H^1(\mathbb{R}^2) \hookrightarrow L^{8/3}(\mathbb{R}^2)$, we deduce 
	\begin{equation}
		\begin{aligned}\label{desigualdadeV2}
			\mathcal{V}_2(u,v)
			&\leq\lambda\int_{\mathbb{R}^2}\int_{\mathbb{R}^2}\frac{u^2(y)u^2(x)+u^2(y)v^2(x)+v^2(y)u^2(x)+v^2(y)v^2(x)}{|x-y|}\dy\dx
			\\
			&\leq C_1\left(\|u\|_{8/3}^4+2\|u\|_{8/3}^2\|v\|_{8/3}^2+\|v\|_{8/3}^4\right)\\
			&\leq C_2\left(\|u\|^4_{W^\lambda}+\|u\|_{W^\lambda}^2\|v\|_{W^\lambda}^2+\|v\|_{W^\lambda}^4\right).
		\end{aligned}
	\end{equation}
        Therefore, $I_\beta$ is well-defined. By standard arguments one may conclude that $I_\beta$ is $C^1$ with	Gateaux derivative
	$$
	\begin{aligned}
		I'_\beta(u,v)(\varphi,\psi)&=\int_{\mathbb{R}^2} \left(\nabla u\nabla\varphi+\nabla v\nabla\psi \right) \dx +\int_{\mathbb{R}^2}\phi_{u,v}\left(u\varphi+v\psi\right)\dx  
		\\
		&\quad- \int_{\mathbb{R}^2}\left(|u|^{2p-2}u+\beta|v|^p|u|^{p-2}u\right)\varphi\dx- \int_{\mathbb{R}^2}\left(|v|^{2p-2}v+\beta|u|^p|v|^{p-2}v\right)\psi\dx,
	\end{aligned}
	$$
	for any $u,v,\varphi,\psi \in W^\lambda$.
	
	\section{Nehari-Pohozaev manifold}\label{Sec-N-P-M}

    In what follows, we recall the Nehari-Pohozev set 
	$$
	\mathcal{M}_{\beta}:=\left\lbrace (u,v)\in \left(W^\lambda\times W^\lambda\right)\setminus\{(0,0)\} : J_{\beta}(u,v)=0 \right\rbrace, 
	$$
	where $J_{\beta}(u,v)$ is defined in \eqref{Juv}. Firstly, we prove that the set $\mathcal{M}_\beta$ is a manifold of $C^1$ class. 
	\begin{proposition}
		The set $\mathcal{M}_{\beta}$ is a $C^1$--manifold. 
	\end{proposition}
	\begin{proof}
		Note that $\mathcal{M}_\beta = J_\beta^{-1}(0)$ and $J_\beta$ is in $C^1$ class. Moreover, for each $(u,v) \in W^\lambda\times W^\lambda$, we mention that 
		\begin{equation}
			J_\beta'(u,v) (u,v) = 4 (\|\nabla u\|_2^2 + \|\nabla v\|_2^2) - (\|u\|^2_2 + \|v\|_2^2)^2 + 4 \mathcal{V}(u,v) - 2(2p -1) \psi_\beta(u,v).
		\end{equation}	
		Let us consider $(u, v) \in \mathcal{M}_\beta$. Hence, 
		\begin{eqnarray}\label{esquema}
			J_\beta'(u,v) (u,v) &=& J_\beta'(u,v) (u,v) - 4 J_\beta (u,v) \nonumber \\
			&\leq&  - 4 (\|\nabla u\|_2^2 + \|\nabla v\|_2^2) + \left[-2(2p -1) + 4 \frac{2p -1}{p}\right]\psi_\beta(u,v)  \nonumber \\
			&<&  -\frac{2(2p -1)(p-2)}{p} \psi_\beta(u,v) \leq 0 
		\end{eqnarray} 
		holds true for each $p \geq 2$. Therefore, by using the Implicit Function Theorem we deduce that $\mathcal{M}_\beta$ is a manifold of $C^1$ class. This ends the proof. 
	\end{proof}
    
    Now, we shall prove that any nonzero function admits exactly one projection over the Nehari-Pohozaev manifold. For this purpose, we need the following auxiliary result:
	
	\begin{lemma}
		Given $u,v \in W^\lambda$, then
		\begin{equation}\label{gradiente}
			\|\nabla \left(t^2u_t\right)\|_2^2=t^4\|\nabla u\|_2^2, \quad \|\nabla \left(t^2v_t\right)\|_2^2=t^4\|\nabla v\|_2^2,
		\end{equation}
		\begin{equation}\label{nao-local3}
			\mathcal{V}(t^2u_t,t^2v_t)=t^4\log(t^{-1})\left(\|u\|_{2}^2+ \|v\|_{2}^2\right)^2+t^{4}\mathcal{V}(u,v),
		\end{equation}
		and
		\begin{equation}\label{l2p}
			\psi_\beta(t^2u_t,t^2v_t)=t^{4p-2}\psi_\beta(u,v),
		\end{equation}
		where $u_t(x):=u(tx)$ and $v_t(x):=v(tx)$, for $t>0$.
	\end{lemma}
	
	\begin{proof}
		Considering $\overline{x}=tx$, we have
		$$
			\int_{\mathbb{R}^2}|\nabla \left(t^2u_t(x)\right)|^2\dx = t^6\int_{\mathbb{R}^2}|\nabla u(tx)|^2\dx = t^4\int_{\mathbb{R}^2}|\nabla u(\overline{x})|^2\,\mathrm{d}\overline{x},
		$$
		and so \eqref{gradiente} holds. Moreover,
		\begin{equation}\label{nao-local}
			\mathcal{V}(t^2u_t,t^2v_t) = t^8\mathcal{V}(u_t,v_t).
		\end{equation}
		Using the change of variable $\overline{y}=ty$, we obtain
		$$
		\begin{aligned}
			\mathcal{V}(u_t,v_t)&=\int_{\mathbb{R}^2}\int_{\mathbb{R}^2}\log(|x-y|)\left(u^2(ty)+ v^2(ty)\right)\left(u^2(tx)+v^2(tx)\right)\dy \dx
			\\
			&= \int_{\mathbb{R}^2}\left(u^2(tx)+v^2(tx)\right)\left(\int_{\mathbb{R}^2}\log(|x-y|)\left(u^2(ty)+v^2(ty)\right)\dy\right)\dx
			\\
			&=t^{-2}\int_{\mathbb{R}^2}\left(u^2(tx)+v^2(tx)\right)\left(\int_{\mathbb{R}^2}\log(|x-t^{-1}\overline{y}|)\left(u^2(\overline{y})+v^2(\overline{y})\right)\,\mathrm{d}\overline{y}\right)\dx,
		\end{aligned}
		$$
		which implies that
		$$
		\mathcal{V}(u_t,v_t) =t^{-4} \int_{\mathbb{R}^2}\int_{\mathbb{R}^2}\log(|t^{-1}\overline{x}-t^{-1}\overline{y}|)\left(u^2(\overline{y})+v^2(\overline{y})\right)\left(u^2(\overline{x})+v^2(\overline{x})\right)\,\mathrm{d}\overline{y}\,\mathrm{d}\overline{x}.
		$$
		Since $\log(|t^{-1}\overline{x}-t^{-1}\overline{y}|) = \log(t^{-1}) + \log(|\overline{x}-\overline{y}|)$, then
		\begin{equation*}\label{nao-local2}
			\mathcal{V}(u_t,v_t) = t^{-4}\log(t^{-1})\left(\|u\|_{2}^2+ \|v\|_{2}^2\right)^2 + t^{-4}\mathcal{V}(u,v).
		\end{equation*}
		Hence, by \eqref{nao-local}, we obtain \eqref{nao-local3}.
	
	Finally, noting that 
	\begin{equation}\label{mud-var-norma-2p}
		\int_{\mathbb{R}^2}|t^2u_t(x)|^{2p}\dx = t^{4p}\int_{\mathbb{R}^2}|u(tx)|^{2p}\dx = t^{4p-2}\int_{\mathbb{R}^2}|u(\overline{x})|^{2p}\,\mathrm{d}\overline{x}
	\end{equation}
	and
	$$
	\int_{\mathbb{R}^2}|t^2u_t(x)t^2v_t(x)|^p\dx = t^{4p}\int_{\mathbb{R}^2}|u(tx)v(tx)|^p\dx = t^{4p-2}\int_{\mathbb{R}^2}|u(\overline{x})v(\overline{x})|^p\,\mathrm{d}\overline{x},
	$$
	we conclude that \eqref{l2p} is valid and the proof is completed.	
	\end{proof}    
    
    Now, we are able to prove the following result:

	\begin{lemma}\label{projecao}
		For each $(u,v)\in \left(W^\lambda\times W^\lambda \right)\setminus\{(0,0)\}$, there exists a unique $t_0=t_0(\beta,u,v)>0$ such that $(t_0^2u_{t_0},t_0^2v_{t_0}) \in \mathcal{M}_{\beta}$, with 
		\begin{equation}\label{maximo}
			I_{\beta}(t_0^2u_{t_0},t_0^2v_{t_0}) = \max_{t>0}I_{\beta}(t^2u_{t},t^2v_{t}).
		\end{equation}
		Furthermore, if $J_{\beta}(u,v)\leq 0$, then $t_0\in(0,1]$.
	\end{lemma}
	
	\begin{proof}
		Consider the function $f:(0,\infty)\to\mathbb{R}$ given by $f(t)=I_{\beta}(t^2u_t,t^2v_t)$.	 We claim that
		\begin{equation}\label{relacao-projecao}
			f'(t)=0 \Longleftrightarrow J_{\beta}(t^2u_{t},t^2v_{t})=0.
		\end{equation}
		In fact, it follows from \eqref{Juv} that
		$$
		\begin{aligned}
			J_{\beta}(t^2u_{t},t^2v_{t}) &= 2\left(\|\nabla (t^2u_t)\|_{2}^2+\|\nabla (t^2v_t)\|_{2}^2\right) - \frac{1}{4}\left(\|t^2u_t\|_{2}^2+ \|t^2v_t\|_{2}^2\right)^2 + \mathcal{V}(t^2u_t,t^2v_t) 
			\\
			&\quad- \frac{4p-2}{2p}\psi_\beta(t^2u_t,t^2v_t).
		\end{aligned}
		$$
		Noting that
		\begin{equation}\label{norma-L2-mud-var}
			\int_{\mathbb{R}^2}\left(t^2u_t(x)\right)^2\dx = t^4  \int_{\mathbb{R}^2}u^2(tx)\dx = t^2\int_{\mathbb{R}^2}u^2(\overline{x})\,\mathrm{d}\overline{x},
		\end{equation}
		we can use \eqref{gradiente}-\eqref{l2p} to deduce
		\begin{equation}\label{J-t2ut-t2vt}
			\begin{aligned}
				J_{\beta}(t^2u_{t},t^2v_{t}) &= 2t^4\left(\|\nabla u\|_{2}^2+\|\nabla v\|_{2}^2\right)-\frac{t^4}{4}\left(\|u\|_{2}^2+ \|v\|_{2}^2\right)^2 
				\\
				&\quad+  t^4\log(t^{-1})\left(\|u\|_{2}^2+ \|v\|_{2}^2\right)^2+ t^{4}\mathcal{V}(u,v)
				-\frac{(4p-2)}{2p}t^{4p-2}\psi_\beta(u,v).
			\end{aligned}
		\end{equation}
		In view of \eqref{gradiente}-\eqref{l2p} and \eqref{norma-L2-mud-var}, we obtain
		\begin{equation}\label{funcional-utvt}
			\begin{aligned}
				f(t)&=I_{\beta}\left(t^2u_t,t^2v_t\right) \\
                    & = \frac{t^4}{2}\left(\|\nabla u\|_{2}^2+\|\nabla v\|_{2}^2\right) + \frac{t^4\log(t^{-1})}{4}\left(\|u\|_{2}^2+ \|v\|_{2}^2\right)^2 
				\\
				&\quad+ \frac{t^{4}}{4}\mathcal{V}(u,v)-\frac{t^{4p-2}}{2p}\psi_\beta(u,v).
			\end{aligned}
		\end{equation}
		Thus, we can conclude that
		\begin{equation}\label{derivada}
			\begin{aligned}
				f'(t) &= 2t^3\left(\|\nabla u\|_{2}^2+\|\nabla v\|_{2}^2\right)+ \frac{t^3\left(4\log(t^{-1})-1\right)}{4}\left(\|u\|_{2}^2+ \|v\|_{2}^2\right)^2
				+  t^3\mathcal{V}(u,v)
				\\
				&\quad-\frac{(4p-2)t^{4p-3}}{2p}\psi_\beta(u,v).
			\end{aligned}
		\end{equation}
		Combining this expression with \eqref{J-t2ut-t2vt}, we obtain
		\begin{equation}\label{igualdade-J-f}
			J_{\beta}(t^2u_{t},t^2v_{t}) = tf'(t),
		\end{equation}
		and hence \eqref{relacao-projecao}	holds.

		Next, we show that there exists $t_1>0$ such that $f(t)>0$, for any $t\in(0,t_1)$ and $\displaystyle\lim_{t\to\infty}f(t)=-\infty$.  Since $f(t)=I_{\beta}(t^2u_t,t^2v_t)$, from \eqref{funcional-utvt}, we infer
		$$
		f(t) = t^4\log(t^{-1})g(t),
		$$
		where
		$$
		\begin{aligned}
			g(t)&:=\frac{1}{2\log(t^{-1})}\left(\|\nabla u\|_{2}^2+\|\nabla v\|_{2}^2\right) + \frac{1}{4}\left(\|u\|_{2}^2+ \|v\|_{2}^2\right)^2 + \frac{1}{4\log(t^{-1})}\mathcal{V}(u,v)
			\\
			&\quad-\frac{t^{4p-6}}{2p\log(t^{-1})}\psi_\beta(u,v).
		\end{aligned}
		$$
		Taking $t\to0^+$, we can conclude that $1/\log(t^{-1})\to0$ and $t^{4p-6}/\log(t^{-1})\to0$, since $p\geq2>3/2$. Hence, there exists $t_1\in(0,1)$ such that $g(t)>0$, for any $t\in(0,t_1)$. The last statement implies that $f(t)>0$, for any $t\in(0,t_1)$. On other hand, by using the fact that $p\geq2>3/2$ and \eqref{funcional-utvt}, we have the convergence $\displaystyle\lim_{t\to\infty}f(t)=-\infty$. So, since $f(0)=0$, then the function $f$  achieves its maximum value at some $t_0=t_0(\beta,u,v)>0$. In particular, $f'(t_0)=0$. Thus, using \eqref{relacao-projecao}, the pair $(t_0^2u_{t_0},t_0^2v_{t_0})$ belongs to $\mathcal{M}_{\beta}$, with 
		$$
		I_{\beta}(t_0^2u_{t_0},t_0^2v_{t_0})=f(t_0)=\max_{t>0}f(t)=  \max_{t>0}I_{\beta}(t^2u_{t},t^2v_{t}).
		$$
		
		Now, we prove that $t_0$  is the unique critical point of $f$. In fact, by \eqref{derivada}, a direct calculation gives
		$$
		\begin{aligned}
			\frac{f'(t)}{t^3} &= 2\left(\|\nabla u\|_{2}^2+\|\nabla v\|_{2}^2\right) + \frac{\left(4\log(t^{-1})-1\right)}{4}\left(\|u\|_{2}^2+ \|v\|_{2}^2\right)^2 + \mathcal{V}(u,v)
			\\
			&\quad-\frac{(4p-2)t^{4p-6}}{2p}\psi_\beta(u,v).
		\end{aligned}
		$$
		Since $p\geq2>3/2$, then $f'(t)/t^3$ is decreasing. Suppose by contradiction that there exists $\widetilde{t}>0$ such that $f'(\widetilde{t})=0$. Assuming without loss of generality $\widetilde{t}<t_0$, we obtain
		$$
		0=\frac{f'(t_0)}{t_0^3}<\frac{f'(\widetilde{t})}{\widetilde{t}^3}=0,
		$$
		which is a contradiction, proving the uniqueness.
		
		Finally, we will ensure that $t_0\in(0,1]$, whenever $J_{\beta}(u,v)\leq 0$. From \eqref{igualdade-J-f},
		$$
		f'(1) = J_{\beta}(u,v) \leq0.
		$$
			Therefore, $t_0\in(0,1]$. The proof is completed.
		\end{proof}

		In view of Lemma~\ref{projecao}, $\mathcal{M}_\beta$ is not empty. Furthermore, the functional $I_\beta$ is bounded from below on $\mathcal{M}_\beta$. Indeed, for each $(u,v) \in \mathcal{M}_{\beta}$, it follows from \eqref{Juv} that
		\begin{equation}\label{limitacao-inferior}
			I_{\beta}(u,v) = I_{\beta}(u,v)-\frac{1}{4}J_{\beta}(u,v) = \frac{1}{8}\left(\|u\|_{2}^2+ \|v\|_{2}^2\right)^2 + \frac{(2p-3)}{4p}\psi_\beta(u,v).
		\end{equation}
		Since $p\geq2$, then $I_{\beta}(u,v)>0$. Hence, $c_{\beta}\in[0,\infty)$. Now, we consider the following technical result:
			
			\begin{lemma}\label{convergencia-t-1}
				For any $u\in H^1(\mathbb{R}^2)$, it holds that
				\begin{equation}\label{convergencia-t-1-norma-H1}
					\lim_{t \to 1}\|t^2{u}_t-{u}\|_{H^1(\mathbb{R}^2)}=0.
				\end{equation}
			\end{lemma}

			\begin{proof}
				Let $\varepsilon>0$ be arbitrary. Since ${u} \in H^1(\mathbb{R}^2)$, by density,  there exist $U:\mathbb{R}^2 \to \mathbb{R}^2$, with $U$ smooth, and $v \in C^\infty(\mathbb{R}^2)$ such that
				\begin{equation}\label{des-norma-densidade-epsolon}
					\|\nabla{u}-U\|_2^2<\varepsilon \quad \mbox{and} \quad \|{u}-v\|_2^2<\varepsilon.
				\end{equation}
				Notice that
				$$
				\int_{\mathbb{R}^2}|\nabla \left(t^2{u}_t\right)-\nabla{u}|^2\dx = \int_{\mathbb{R}^2}\left| \left(\nabla \left(t^2{u}_t\right)-U\right)+\left(U-\nabla{u}\right)\right| ^2\dx.
				$$
				Using the inequality $(a+b)^2 \leq 2a^2+2b^2, \  \forall a,b\geq0$, with \eqref{des-norma-densidade-epsolon}, it holds that
				$$
				\int_{\mathbb{R}^2}|\nabla \left(t^2{u}_t\right)-\nabla{u}|^2\dx \leq 2\int_{\mathbb{R}^2}|t^3\nabla {u}(tx)-U(x)|^2\dx+2\varepsilon.
				$$
				By writing
				$$
				|t^3\nabla {u}(tx)-U(x)|^2 = \left|\left(t^3\nabla {u}(tx)-t^3U(tx)\right)+\left(t^3U(tx)-U(tx)\right)+\left(U(tx)-U(x)\right) \right|^2,
				$$
				we can use the estimate $(a+b+c)^2 \leq 3a^2+3b^2+3c^2, \  \forall a,b,c\geq0$, to obtain
				$$
				\begin{aligned}
					\int_{\mathbb{R}^2}|\nabla \left(t^2{u}_t\right)-\nabla{u}|^2\dx &\leq 6t^6\int_{\mathbb{R}^2}|\nabla {u}(tx)-U(tx)|^2\dx 
					\\
					&\quad+ 6|t^3-1|^2\int_{\mathbb{R}^2}|U(tx)|^2\dx+6\int_{\mathbb{R}^2}|U(tx)-U(x)|^2\dx.
				\end{aligned}
				$$
				Applying the change variables $y=tx$, we have that
				$$
				\begin{aligned}
					\int_{\mathbb{R}^2}|\nabla \left(t^2{u}_t\right)-\nabla{u}|^2\dx &\leq 6t^4\int_{\mathbb{R}^2}|\nabla {u}(y)-U(y)|^2\dy+6|t^3-1|^2t^{-2}\int_{\mathbb{R}^2}|U(y)|^2\dy
					\\
					&\quad +6\int_{\mathbb{R}^2}|U(tx)-U(x)|^2\dx,
				\end{aligned}
				$$
				which jointly with \eqref{des-norma-densidade-epsolon} imply that
				$$
				\int_{\mathbb{R}^2}|\nabla \left(t^2{u}_t\right)-\nabla{u}|^2\dx \leq 6\varepsilon+o(1),
				$$
				where $o(1)\to 0$, as $t\to1$.
				Using the same idea, we can conclude that
				$$
				\int_{\mathbb{R}^2}|t^2{u}_t-{u}|^2\dx \leq C_1\varepsilon+o(1),
				$$
				and therefore \eqref{convergencia-t-1-norma-H1} holds, which finishes the proof. 	
				\end{proof}
				
				For the next result, we shall discuss the behavior of the map $t \mapsto I_{\beta}\left(t^2u_t,t^2v_t\right)$.

				\begin{lemma}\label{desigualdade-top-ln}
					If $(u,v) \in \mathcal{M}_\beta$, then
					$$
					I_{\beta}\left(t^2u_t,t^2v_t\right) \leq 	I_\beta(u,v)-\frac{1-t^4+4t^4\log(t)}{16}\left(\|u\|_{2}^2+ \|v\|_{2}^2\right)^2, \quad \forall t>0.
					$$
				\end{lemma}
				\begin{proof}
					In view of \eqref{funcional-utvt}, we have
					$$
					\begin{aligned}
						I_\beta(u,v)-I_{\beta}\left(t^2u_t,t^2v_t\right)&=\frac{1-t^4}{2}\left(\|\nabla u\|_{2}^2+\|\nabla v\|_{2}^2\right)+\frac{1-t^4}{4}\mathcal{V}(u,v)
						\\
						&\quad-\frac{t^4\log(t^{-1})}{4}\left(\|u\|_{2}^2+ \|v\|_{2}^2\right)^2+\left(-\frac{1}{2p}+\frac{t^{4p-2}}{2p}\right)\psi_\beta(u,v).
					\end{aligned}
					$$
					On the other hand, by \eqref{Juv} one deduces
					$$
					\begin{aligned}
						0&=\frac{1-t^4}{4}J_{\beta}(u,v) \\
                        &= \frac{1-t^4}{2}\left(\|\nabla u\|_{2}^2+\|\nabla v\|_{2}^2\right)-\frac{1-t^4}{16}\left(\|u\|_{2}^2+ \|v\|_{2}^2\right)^2
						\\
						&\quad+\frac{1-t^4}{4}\mathcal{V}(u,v)-\frac{(1-t^4)}{4}\frac{(4p-2)}{2p}\psi_\beta(u,v).
					\end{aligned}
					$$
					Hence,
					$$
					\begin{aligned}
						\frac{1-t^4}{2}\left(\|\nabla u\|_{2}^2+\|\nabla v\|_{2}^2\right)+\frac{1-t^4}{4}\mathcal{V}(u,v)&=\frac{1-t^4}{16}\left(\|u\|_{2}^2+ \|v\|_{2}^2\right)^2
						\\
						&\quad+\frac{(1-t^4)}{4}\frac{(4p-2)}{2p}\psi_\beta(u,v).
					\end{aligned}
					$$
					This expression combined with the first equality, infer that
					\begin{equation}\label{subtracao-I-It2}
						\begin{aligned}
							I_\beta(u,v)-I_{\beta}\left(t^2u_t,t^2v_t\right)&=\frac{1-t^4+4t^4\log(t)}{16}\left(\|u\|_{2}^2+ \|v\|_{2}^2\right)^2
							\\
							&\quad+\left[\frac{(1-t^4)}{4}\frac{(4p-2)}{2p}-\frac{1}{2p}+\frac{t^{4p-2}}{2p}\right]\psi_\beta(u,v).
						\end{aligned}
					\end{equation}
					Consider the function $g:(0,\infty) \to \mathbb{R}$ given by
					$$
					g(t):=\frac{(1-t^4)}{4}\frac{(4p-2)}{2p}-\frac{1}{2p}+\frac{t^{4p-2}}{2p}=\frac{1}{8p}\left((1-t^4)(4p-2)-4+4t^{4p-2}\right).
					$$
					Note that
					$$
					g'(t)=\frac{1}{8p}\left(-4t^3(4p-2)+4(4p-2)t^{4p-3}\right).
					$$
					Since $p\geq 2 \geq3/2$, it follows that $g'(t)>0$, if $t>1$ and $g'(t)<0$, when $t<1$. As $g(1)=0$, we get $g(t)\geq0$, for any $t>0$, because $1$ is a minimum point of $g$. This condition together with \eqref{subtracao-I-It2} complete the proof.
					\end{proof}
					
			\begin{remark}\label{funcao-positiva} Consider the function $h: (0 , \infty)  \to [0, \infty)$ given by $h(t):=1-t^4+4t^4\log(t)$, $t > 0$. Notice that
				$$
				h(t)\geq0, \quad \forall t>0.
				$$
				In fact, we observe that
				$$
				h'(t)=-4t^3+16t^3\log(t)+4t^3=16t^3\log(t).
				$$
				Consequently, $h'(t)<0$, for $t<1$ and $h'(t)>0$, if $t>1$. Hence, $h(t) \geq 0$ for each $t > 0$, since $h(1)=0$ and 1 is a minimum point of $h$. In particular, we infer that 
				\begin{equation*}\label{h-sem-1}
					h(t) = 1-t^4+4t^4\log(t)>0, \quad \forall t>0 \ (t\neq1).
				\end{equation*}
			\end{remark}
					
    \section{Proof of Theorem~\ref{existencia-de-solucao}}\label{Sec-Teo-existencia}
		
		In order to show that $\mathcal{M}_\beta$ is a natural constraint, the Quantitative Deformation Lemma \cite[Lemma 2.3]{Willem} plays a crucial role in our arguments. Namely, we consider the following result:

					\begin{lemma}[Quantitative Deformation Lemma]\label{lema-deformacao}
						Let $X$ be a Banach space, $I \in C^1(X,\mathbb{R})$, $S \subset X$, $c \in \mathbb{R}$,  $\varepsilon,\delta>0$ such that
						$$
						\|I'(u)\|\geq \frac{8\varepsilon}{\delta}, \quad\forall u \in I^{-1}\left([c-2\varepsilon,c+2\varepsilon]\right)\cap S_{2\delta},
						$$
						where 
						$$
						S_{2\delta}:=\left\lbrace u \in X:d(u,S):=\inf_{y \in S}\|u-y\|\leq 2\delta\right\rbrace.
						$$
						Then, there exists $\eta \in C\left([0,1]\times X,X\right)$ such that
						\begin{itemize}
							\item [$(i)$] $\eta(t,u)=u$, if $t=0$ or if $u \notin I^{-1}\left([c-2\varepsilon,c+2\varepsilon]\right)\cap S_{2\delta}$;
							
							\item [$(ii)$] $\eta\left(1,I^{c+\varepsilon}\cap S\right) \subset I^{c-\varepsilon}$, where $I^d:=\{u \in X: I(u)\leq d\}$, $\forall d \in \mathbb{R}$;
							
							\item [$(iii)$] $\eta(t,\cdot)$ is an homeomorphism of $X$, for each $t \in [0,1]$;
							
							\item[$(iv)$] $\|\eta(t,u)-u\|\leq \delta$, $\forall u \in X$, $\forall t \in[0,1]$;
							
							\item[$(v)$] $I\left(\eta(\cdot,u)\right)$ is non increasing, for each $u \in X$;
							
							\item[$(vi)$] $I\left(\eta(t,u)\right)<c$, $\forall u \in I^c\cap S_\delta$, $\forall t \in (0,1]$.
						\end{itemize}
					\end{lemma}
					
					In what follows, we shall prove that any critical point for the functional $I_\beta$ restricted to $\mathcal{M}_\beta$ is a critical point for $I_\beta$. This can be summarized as follows:

					\begin{proposition}\label{esquema}
						If $(\overline{u},\overline{v}) \in \mathcal{M}_\beta$ is such that $I_\beta(\overline{u},\overline{v})=c_\beta$, then $I_\beta'(\overline{u},\overline{v})=0$.
					\end{proposition}
					
					\begin{proof}
						Considering $\overline{w}=(\overline{u},\overline{v})$, suppose by contradiction that  $I_\beta'(\overline{w})\neq0$. Thus, there exist $\delta,\rho>0$ such that
						\begin{equation}\label{implicacao-derivada}
							\|w-\overline{w}\|_{H^1(\mathbb{R}^2) \times H^1(\mathbb{R}^2)}\leq 3\delta \quad \Longrightarrow \quad \|I'_\beta(u,v)\|\geq \rho.
						\end{equation}
						Recall that if $w=(u,v) \in H^1(\mathbb{R}^2) \times H^1(\mathbb{R}^2)$, then 
						$$
						\|w-\overline{w}\|_{H^1(\mathbb{R}^2) \times H^1(\mathbb{R}^2)}=\|u-\overline{u}\|_{H^1(\mathbb{R}^2)}+\|v-\overline{v}\|_{H^1(\mathbb{R}^2)}.
						$$
			Define
			$$
			S:=\overline{B(\overline{w},\delta)}:=\{w=(u,v) \in W^\lambda\times W^\lambda:	\|w-\overline{w}\|_{H^1(\mathbb{R}^2) \times H^1(\mathbb{R}^2)}\leq \delta\}.
			$$ 			
			Let $w \in I_\beta^{-1}\left([c_\beta-2\varepsilon,c_\beta+2\varepsilon]\right)\cap S_{2\delta}$, with $\varepsilon>0$ to be chosen later.	Let $y \in S$ be such that $\|w-y\|_{H^1(\mathbb{R}^2) \times H^1(\mathbb{R}^2)}=d(w,S)\leq 2\delta$. Thus,		
					$$
					\|w-\overline{w}\|_{H^1(\mathbb{R}^2) \times H^1(\mathbb{R}^2)} \leq \|\overline{w}-y\|_{H^1(\mathbb{R}^2) \times H^1(\mathbb{R}^2)} + \|y-w\|_{H^1(\mathbb{R}^2) \times H^1(\mathbb{R}^2)} \leq \delta+2\delta=3\delta. 
					$$
						Choosing $0<\varepsilon\leq \rho \delta/8$, we can apply \eqref{implicacao-derivada} to get
						$$
						\|I_\beta'(w)\|\geq\rho\geq \frac{8\varepsilon}{\delta}, \quad\forall w \in I_\beta^{-1}\left([c_\beta-2\varepsilon,c_\beta+2\varepsilon]\right)\cap S_{2\delta}.
						$$
						Therefore, we are able to use the items of Lemma~\ref{lema-deformacao}.
						
						Next, we apply Lemma~\ref{convergencia-t-1} to obtain $\delta_1>0$ such that
						$$
						0<|t-1|<\delta_1 \quad \Longrightarrow \quad \|t^2\overline{w}_t-\overline{w}\|_{H^1(\mathbb{R}^2) \times H^1(\mathbb{R}^2)}<\delta \ \ (t^2\overline{w}_t \in S),
						$$
						where $\overline{w}_t=(\overline{u}_t,\overline{v}_t)$. 
						By Lemma~\ref{desigualdade-top-ln} and Remark~\ref{funcao-positiva},
						$$
						I_\beta(t^2\overline{w}_t) \leq I_\beta(\overline{w})=c_\beta\leq c_\beta+\varepsilon,
						$$
				that is $t^2\overline{w}_t \in I_\beta^{c_\beta+\varepsilon}$. Thus, using item $(ii)$ of Lemma~\ref{lema-deformacao}, it follows that
						\begin{equation}\label{implicacao-menor-que-1}
							0<|t-1|<\delta_1 \quad \Longrightarrow \quad I_\beta\left(\eta(1,t^2\overline{w}_t)\right)\leq c_\beta-\varepsilon.
						\end{equation}
						
					On other hand, combining items $(v)$ and $(i)$ of Lemma~\ref{lema-deformacao} together with Lemma~\ref{desigualdade-top-ln}, there holds	
					$$
					\begin{aligned}
						I_\beta\left(\eta(1,t^2\overline{w}_t)\right) & \leq I_\beta\left(\eta(0,t^2\overline{w}_t)\right)\\ &= I_\beta(t^2\overline{w}_t) 
						\\
						&\leq I_\beta(\overline{w})-\frac{1-t^4+4t^4\log(t)}{16}\left(\|\overline{u}\|_{2}^2+ \|\overline{v}\|_{2}^2\right)^2.
					\end{aligned}
					$$
					In view of Remark~\ref{funcao-positiva},
					$$
					|t-1|\geq\delta_1 \quad \Longrightarrow \quad 0<\xi:=\min_{s \in (1-\delta_1,1+\delta_1) \setminus\{1\}}\{h(s)\}\leq h(t).
					$$ 
					As consequence,
					$$
					|t-1|\geq\delta_1 \quad \Longrightarrow \quad I_\beta\left(\eta(1,t^2\overline{w}_t)\right) < c_\beta-\frac{\xi}{16}\left(\|\overline{u}\|_{2}^2+ \|\overline{v}\|_{2}^2\right)^2.
					$$
					By taking $\varepsilon=\min\{\rho \delta/8,(\xi/32)\left(\|\overline{u}\|_{2}^2+ \|\overline{v}\|_{2}^2\right)^2\}$, we can conclude that
					\begin{equation}\label{implicacao-maior-que-1}
						|t-1|\geq\delta_1 \quad \Longrightarrow \quad I_\beta\left(\eta(1,t^2\overline{w}_t)\right) < c_\beta-2\varepsilon.
					\end{equation}
					
					Now, recalling that in Lemma~\ref{projecao}, since $(\overline{u},\overline{v}) \in \mathcal{M}_\beta$, the function $f(t)=I_{\beta}(t^2\overline{u}_t,t^2\overline{v}_t)$ satisfy $f'(1)=0$, $f'(t)>0$ for $t\in(0,1)$ and $f'(t)<0$ for $t\in(1,\infty)$, by \eqref{igualdade-J-f} we get
					\begin{equation}\label{des-J-beta-t1-t2}
						J_{\beta}(t_1^2\overline{w}_{t_1})>0 \quad \mbox{and}  \quad J_{\beta}(t_2^2\overline{w}_{t_2})<0,
					\end{equation}
					with $t_1\in(0,1)$ and $t_2\in(1,\infty)$ such that $|t_1-1|,|t_2-1|\geq\delta_1$. Now, notice that combining \eqref{implicacao-menor-que-1} with \eqref{implicacao-maior-que-1}, we have
					\begin{equation}\label{maximo-c}
						\max_{t \in [t_1,t_2]}I_\beta\left(\eta(1,t^2\overline{w}_t)\right)<c_\beta.
					\end{equation}
					Finally, define $\Psi(t):=J_\beta\left(\eta(1,t^2\overline{w}_t)\right)$, for any $t>0$. Using \eqref{implicacao-maior-que-1} and item $(i)$ of Lemma~\ref{lema-deformacao}, $\eta(1,t_1^2\overline{w}_{t_1})=t_1^2\overline{w}_{t_1}$ and $\eta(1,t_1^2\overline{w}_{t_2})=t_2^2\overline{w}_{t_2}$. From \eqref{des-J-beta-t1-t2},
					$$
					\Psi(t_2)<0<\Psi(t_1).
					$$
					Since $\Psi$ is continuous, there exists $r\in(t_1,t_2)$ such that $\Psi(r)=J_\beta\left(\eta(1,r^2\overline{w}_r)\right)=0$.  Consequently, $\eta(1,r^2\overline{w}_r) \in \mathcal{M}_\beta$. So, by \eqref{maximo-c} we obtain
					$$
					c_\beta \leq I_\beta\left(\eta(1,r^2\overline{w}_r)\right)\leq\max_{t \in [t_1,t_2]}I_\beta\left(\eta(1,t^2\overline{w}_t)\right)<c_\beta, 
					$$
					which is a contradiction. The proof is completed.
					\end{proof}

	\subsection{Existence of solution} Now, we prove that any minimizer sequence for $I_\beta$ restricted to $\mathcal{M}_\beta$ is bounded. This can be done thanks to Gagliardo-Nirenberg inequality together with some fine estimates. Specifically, we are able to prove the following result: 

	\begin{lemma}\label{limitacao-H1}
	If $\{(u_n,v_n)\}_n\subset\mathcal{M}_\beta$ is a minimizing
	sequence for $c_\beta$, then the sequences $\{u_n\}_n$ and $\{v_n\}_n$ are bounded in the norm $\|\cdot\|_{H^1(\mathbb{R}^2)}$.
\end{lemma}

	\begin{proof}
	Firstly, by using  \eqref{limitacao-inferior}, we infer that
	\begin{equation}\label{estimativa-sequencia}
		c_\beta + o_n(1) = I_\beta(u_n,v_n) \geq \frac{1}{8}\left(\|u_n\|_{2}^2+ \|v_n\|_{2}^2\right)^2+ \frac{2p-3}{4p}\psi_\beta(u_n,v_n),
	\end{equation}
	where $o_n(1)\to 0$, as $n\to\infty$. Since $p\geq2\geq3/2$, we can conclude that  $\{\|u_n\|_2\}_n$ and $\{\|v_n\|_2\}_n$ are bounded.
	
	Now, we prove that $\{\|\nabla u_n\|_2\}_n$ and $\{\|\nabla v_n\|_2\}_n$ are also bounded. Using \eqref{Juv}, we can compute
	\begin{equation}\label{igualdade-cbeta}
		\begin{aligned}
			c_\beta + o_n(1) &= I_\beta(u_n,v_n)-\frac{1}{8}J_\beta(u_n,v_n) \\
			&= \frac{1}{4}\left(\|\nabla u_n\|_{2}^2+\|\nabla v_n\|_{2}^2\right)+\frac{1}{32}\left(\|u_n\|_{2}^2+ \|v_n\|_{2}^2\right)^2
			+ \frac{1}{8}\mathcal{V}(u_n,v_n) 
			\\
			&\quad+ \frac{(4p-10)}{16p}\psi_\beta(u_n,v_n).
		\end{aligned}
	\end{equation}
	Consequently,
	$$
	c_\beta + o_n(1) + \frac{(10-4p)}{16p}\psi_\beta(u_n,v_n) \geq \frac{1}{4}\left(\|\nabla u_n\|_{2}^2+\|\nabla v_n\|_{2}^2\right)  + \frac{1}{8}\mathcal{V}(u_n,v_n).
	$$
From \eqref{estimativa-sequencia}, the sequence $\{\psi_\beta(u_n,v_n)\}_n$ is bounded. Thus, there exists $C_0>0$ such that for $n\in\mathbb{N}$ large,
$$
C_0+c_\beta \geq \|\nabla u_n\|_{2}^2+\|\nabla v_n\|_{2}^2  + \mathcal{V}(u_n,v_n).
$$
This estimate combined with \eqref{decomposicao} and the fact that $\mathcal{V}_1(u_n,v_n)\geq0$, imply that
\begin{equation}\label{gradiente-termo-nao-local}
	C_0+c_\beta \geq \|\nabla u_n\|_{2}^2+\|\nabla v_n\|_{2}^2 -\mathcal{V}_2(u_n,v_n),
\end{equation}
for $n\in\mathbb{N}$ large. To estimate the term $\mathcal{V}_2(u_n,v_n)$, we apply the Gagliardo-Nirenberg inequality, namely
\begin{equation}\label{Ga-Ni}
	\|u\|_s^s \leq C_s\|u\|_2^2\|\nabla u\|_2^{s-2},
\end{equation}
for $s>2$ and for $u \in H^1(\mathbb{R}^2)$. By \eqref{desigualdadeV2},
$$
\mathcal{V}_2(u_n,v_n) \leq C_1\left(\|u_n\|_{8/3}^4+2\|u_n\|_{8/3}^2\|v_n\|_{8/3}^2+\|v_n\|_{8/3}^4\right).
$$
Using \eqref{Ga-Ni} with $s=8/3$, there holds
$$
\|u_n\|_{8/3} \leq C_2\|u_n\|_2^{3/4}\|\nabla u_n\|_2^{1/4} \quad \mbox{and} \quad \|v_n\|_{8/3} \leq C_2\|v_n\|_2^{3/4}\|\nabla v_n\|_2^{1/4},
$$
for any $n\in\mathbb{N}$. Thus,
$$
\mathcal{V}_2(u_n,v_n) \leq C_3\left(\|u_n\|_2^{3}\|\nabla u_n\|_2 + \|u_n\|_2^{3/2}\|\nabla u_n\|_2^{1/2}\|v_n\|_2^{3/2}\|\nabla v_n\|_2^{1/2} + \|v_n\|_2^{3}\|\nabla v_n\|_2\right).
$$
By Young's inequality,
$$
\|u_n\|_2^{3/2}\|\nabla u_n\|_2^{1/2}\|v_n\|_2^{3/2}\|\nabla v_n\|_2^{1/2} \leq \frac{\|u_n\|_2^{3}\|\nabla u_n\|_2}{2} + \frac{\|v_n\|_2^{3}\|\nabla v_n\|_2}{2}.
$$
	As consequence,
$$
\mathcal{V}_2(u_n,v_n) \leq C_4\left(\|u_n\|_2^{3}\|\nabla u_n\|_2  + \|v_n\|_2^{3}\|\nabla v_n\|_2\right).
$$	
Now, let $\varepsilon>0$ be arbitrary. Using Young's inequality once more, we deduce
$$
\|u_n\|_2^{3}\|\nabla u_n\|_2 \leq C(\varepsilon)\|u_n\|_2^6 + \varepsilon\|\nabla u_n\|_2^2\quad \mbox{and} \quad 	\|v_n\|_2^{3}\|\nabla v_n\|_2 \leq C(\varepsilon)\|v_n\|_2^6 + \varepsilon\|\nabla v_n\|_2^2.
$$ 
Hence,  
\begin{equation}\label{est-V2-unvn}
	\mathcal{V}_2(u_n,v_n) \leq C_5\left(\|u_n\|_2^6 + \|v_n\|_2^6\right) + C_4\varepsilon\left(\|\nabla u_n\|_2^2 + \|\nabla v_n\|_2^2\right).
\end{equation}
Combining this estimate with \eqref{gradiente-termo-nao-local}, one conclude that  
\begin{equation}\label{limitacao-gradiente}
	C_0+c_\beta \geq \left(1-{C_4\varepsilon} \right)\left(\|\nabla u_n\|_{2}^2+\|\nabla v_n\|_{2}^2\right) -{C_5}\left(\|u_n\|_2^6 + \|v_n\|_2^6\right),
\end{equation}
for $n\in\mathbb{N}$ large. Picking $\varepsilon>0$ small, such that $1-C_4\varepsilon>0$ and using that $\{\|u_n\|_2\}_n$ and $\{\|v_n\|_2\}_n$ are bounded, the last estimate allows us to ensure that $\{\|\nabla u_n\|_2\}_n$ and $\{\|\nabla v_n\|_2\}_n$ are bounded.
\end{proof}

Next, we shall prove that the Nehari-Pohozaev manifold $\mathcal{M}_\beta$ is away from the origin. This is important in order to avoid minimizer sequences that weakly converges to the zero function. Precisely, we prove the following result:

	\begin{lemma}\label{longe-do-zero}
	Assume that $\lambda>e^{1/4}$. Then, there exists $\delta=\delta(\beta)>0$ such that
	$$
	\|u\|_{H^1(\mathbb{R}^2)}+\|v\|_{H^1(\mathbb{R}^2)}\geq \delta, \quad \forall (u,v) \in \mathcal{M}_\beta.
	$$
\end{lemma}

	\begin{proof}
	The proof follows arguing by contradiction. Suppose that there
	exists a sequence $\{(u_n,v_n)\}_n \subset \mathcal{M}_\beta$ such that
	\begin{equation}\label{convergencia-zero-norma}
		\|u_n\|_{H^1(\mathbb{R}^2)}+\|v_n\|_{H^1(\mathbb{R}^2)}=o_n(1).
	\end{equation}
Notice that
\begin{equation}\label{identidade-J-sequencia}
	2\left(\|\nabla u_n\|_2^2+\|\nabla v_n\|_2^2\right) + \log(\lambda)\left(\|u_n\|_2^2+\|v_n\|_2^2\right)^2 =	2\left(\|\nabla u_n\|_2^2+\|\nabla v_n\|_2^2\right) + \mathcal{T}(u_n,v_n),
\end{equation}
where
$$
\mathcal{T}(u_n,v_n):=\log(\lambda)\left(\|u_n\|_2^4+2\|u_n\|_2^2\|v_n\|_2^2+\|v_n\|_2^4\right).
$$
Noting that
$$
\int_{\mathbb{R}^2}\int_{\mathbb{R}^2}\log(\lambda+|x-y|)u^2(x)v^2(y)\,\mathrm{d}y\dx \geq \log (\lambda)\|u\|_2^2\|v\|_2^2,
$$
one deduces
$$
\begin{aligned}
	\mathcal{T}(u_n,v_n) &\leq \int_{\mathbb{R}^2}\int_{\mathbb{R}^2}\log(\lambda+|x-y|)u_n^2(x)u_n^2(y)\,\mathrm{d}y\dx 
	\\
	&\quad+ \int_{\mathbb{R}^2}\int_{\mathbb{R}^2}\log(\lambda+|x-y|)u_n^2(x)v_n^2(y)\,\mathrm{d}y\dx 
	\\
	&\quad+ \int_{\mathbb{R}^2}\int_{\mathbb{R}^2}\log(\lambda+|x-y|)v_n^2(x)u_n^2(y)\,\mathrm{d}y\dx
	\\
	&\quad+\int_{\mathbb{R}^2}\int_{\mathbb{R}^2}\log(\lambda+|x-y|)v_n^2(x)v_n^2(y)\,\mathrm{d}y\dx
	\\
	&=\mathcal{V}_1(u_n,v_n).
\end{aligned}
$$
This estimate combined with \eqref{identidade-J-sequencia} infer that
$$
2\left(\|\nabla u_n\|_2^2+\|\nabla v_n\|_2^2\right) + \log(\lambda)\left(\|u_n\|_2^2+\|v_n\|_2^2\right)^2 \leq 2\left(\|\nabla u_n\|_2^2+\|\nabla v_n\|_2^2\right) + \mathcal{V}_1(u_n,v_n).
$$
Since $J_\beta(u_n,v_n)=0$, by \eqref{Juv}, one has
\begin{equation}\label{norma-H1-por-baixo}
	\begin{aligned}
		2\left(\|\nabla u_n\|_2^2+\|\nabla v_n\|_2^2\right) + \log(\lambda)\left(\|u_n\|_2^2+\|v_n\|_2^2\right)^2 &\leq \mathcal{V}_2(u_n,v_n)+ \frac{1}{4}\left(\|u_n\|_2^2+\|v_n\|_2^2\right)^2
		\\
		&\quad+\frac{(4p-2)}{2p}\psi_\beta(u_n,v_n).
	\end{aligned}
\end{equation}
{Combining \eqref{est-V2-unvn} and \eqref{psi} (with H\"older's inequality)}, we get
$$
\begin{aligned}
	2\left(\|\nabla u_n\|_2^2+\|\nabla v_n\|_2^2\right) + \log(\lambda)\left(\|u_n\|_2^2+\|v_n\|_2^2\right)^2 &\leq C_1\left(\|u_n\|_2^6 + \|v_n\|_2^6\right)\\
	&\quad+ C_1\varepsilon\left(\|\nabla u_n\|_2^2 + \|\nabla v_n\|_2^2\right)
	\\
	&\quad+\frac{1}{4}\left(\|u_n\|_2^2+\|v_n\|_2^2\right)^2
	\\
	&\quad +C_1\left(\|u_n\|_{2p}^{2p}+\|v_n\|_{2p}^{2p}\right).
\end{aligned}
$$
As consequence, defining
$$
\alpha_n:=2\left(\|\nabla u_n\|_2^2+\|\nabla v_n\|_2^2\right) + \log(\lambda)\left(\|u_n\|_2^2+\|v_n\|_2^2\right)^2,
$$
we have that
$$
1\leq C_1\frac{\|u_n\|_2^6 + \|v_n\|_2^6}{\alpha_n}+C_1\varepsilon\frac{\|\nabla u_n\|_2^2 + \|\nabla v_n\|_2^2}{\alpha_n}+\frac{1}{4}\frac{\left(\|u_n\|_2^2+\|v_n\|_2^2\right)^2}{\alpha_n}+C_1\frac{\|u_n\|_{2p}^{2p}+\|v_n\|_{2p}^{2p}}{\alpha_n}.
$$
Thus, we see that
$$
1 \leq C_2 \left(\|u_n\|_2^2 + \|v_n\|_2^2\right) + C_2\varepsilon + \frac{1}{4\log(\lambda)}+C_1\frac{\|u_n\|_{2p}^{2p}+\|v_n\|_{2p}^{2p}}{\alpha_n}.
$$
From \eqref{Ga-Ni} together with Young's inequality, we can estimate
$$
\|u_n\|_{2p}^{2p} \leq C_3\|u_n\|_2^2\|\nabla u_n\|_2^{2p-2} \leq C_3\left(\varepsilon\|u_n\|_2^4+C(\varepsilon)\|\nabla u_n\|_2^{2(2p-2)}\right).
$$
Similarly,
$$
\|v_n\|_{2p}^{2p} \leq C_4\|v_n\|_2^2\|\nabla v_n\|_2^{2p-2} \leq C_4\left(\varepsilon\|v_n\|_2^4+C(\varepsilon)\|\nabla v_n\|_2^{2(2p-2)}\right).
$$
The last three estimates, imply that
$$
1 \leq C_2 \left(\|u_n\|_2^2 + \|v_n\|_2^2\right) + C_2\varepsilon + \frac{1}{4\log(\lambda)}+C_5\varepsilon+C_5\left(\|\nabla u_n\|_2^{2(2p-2)-2}+\|\nabla v_n\|_2^{2(2p-2)-2}\right).
$$
Since $p>3/2$, then $2(2p-2)-2>0$, and hence
$$
1\leq C_2\varepsilon+\frac{1}{4\log(\lambda)}+C_5\varepsilon.
$$
Taking $\varepsilon \to 0^+$, we see that  $\lambda \leq e^{1/4}$ which is a  contradiction. 
\end{proof}

Finally, we show that the existence of weak solution for System \eqref{S}.

\begin{proposition}\label{atingibilidade}
	Assume that $2\leq p<\infty$ and $\beta\geq0$. Then, there exists a pair $(u,v)\in\mathcal{M}_\beta$, with $\lambda>e^{1/4}$, such that
	$$
	I_\beta(u,v)=c_\beta, \quad \mbox{with} \ \ u,v\geq0.
	$$
	In particular, $(u,v)$ is a weak solution of \eqref{S}.
\end{proposition}

\begin{proof}
	Let $\{(u_n,v_n)\}_n \subset \mathcal{M}_\beta$ be such that 
	\begin{equation}\label{sequencia-minimizante}
		I_\beta(u_n,v_n)=c_\beta+o_n(1).
	\end{equation}
	
		We divide the proof in four steps.
	\bigskip
	
	\noindent\underline{First step:} It holds that
	\begin{equation}\label{norma-bola-positiva}
		\liminf_{n\to \infty}\left(\sup_{y \in \mathbb{Z}^2}\int_{B_R(y)}\left(u_n^2+v_n^2\right)\dx\right)\geq \eta, \quad \mbox{for some
		} \ R>0,
	\end{equation}
	where $B_R(y):=\{ x \in \mathbb{R}^2 : |x-y|<R\}$. In fact, otherwise, after passing to a subsequence, $u_n,v_n \to 0$ in $L^s(\mathbb{R}^2)$, for any $s>2$, due to Lions’ Lemma, see e.g. \cite[Lemma 1.21]{Willem}. As consequence, we can use \eqref{desigualdadeV2} to deduce
	$$
	\mathcal{V}_2(u_n,v_n)=o_n(1).
	$$
	{On other hand, from \eqref{psi}} and H\"older's inequality, we can conclude that
	$$
	\psi_\beta(u_n,v_n)=o_n(1).
	$$
	Applying these convergences in \eqref{norma-H1-por-baixo}, there holds
	$$
	2\left(\|\nabla u_n\|_2^2+\|\nabla v_n\|_2^2\right) + \frac{\log(\lambda)-1}{4}\left(\|u_n\|_2^2+\|v_n\|_2^2\right)^2\leq o_n(1).
	$$
	Recall that $\lambda> e^{1/4}$. Hence, 
	\begin{equation}\label{seq-min-conv-zero}
		\|u_n\|_{H^1(\mathbb{R}^2)}+\|v_n\|_{H^1(\mathbb{R}^2)}=o_n(1),
	\end{equation}
	which is a contradiction with Lemma~\ref{longe-do-zero} and finishes the proof of the first step.

	\bigskip
	
	\noindent\underline{Second step:} The sequence $\{\mathcal{V}_1(u_n,v_n)\}_n \subset \mathbb{R}$ is bounded.
	
	\bigskip

	Using \eqref{Juv}, we have that
	$$
	\mathcal{V}_1(u_n,v_n) \leq \mathcal{V}_2(u_n,v_n) +  \frac{1}{4}\left(\|u_n\|_{2}^2+ \|v_n\|_{2}^2\right)^2 + \frac{4p-2}{2p}\psi_\beta(u_n,v_n).
	$$
	Since the sequences $\{u_n\}_n$ and $\{v_n\}_n$ are bounded in the norm $\|\cdot\|_{H^1(\mathbb{R}^2)}$ (see Lemma~\ref{limitacao-H1}), we can combine \eqref{est-V2-unvn} {together with \eqref{psi}, H\"older's inequality,} and the embedding $H^1(\mathbb{R}^2)\hookrightarrow L^{2p}(\mathbb{R}^2)$ to ensure that $\{\mathcal{V}_1(u_n,v_n)\}_n \subset \mathbb{R}$ is bounded, completing this step.
	
	\bigskip
	
	\noindent\underline{Third step:} There exists $\{y_n\}_n\subset \mathbb{Z}^2$ such that the sequences $\{\overline{u}_n\}_n$ and $\{\overline{v}_n\}_n$, given by  $\overline{u}_n(x):=u_n(x-y_n)$ and $\overline{v}_n(x):=v_n(x-y_n)$, are bounded in the norm $\|\cdot\|_W^\lambda$.
	
	\bigskip
	
	Firstly, by using \eqref{norma-bola-positiva}, there exists $\{y_n\}_n\subset \mathbb{Z}^2$ such that 
	\begin{equation}\label{des-transladada-nao-nula}
		\int_{B_R(0)}\left(\overline{u}_n^2+\overline{v}_n^2\right)\dx \geq \frac{\eta}{2}.
	\end{equation}
	Since $I_\beta$ and $J_\beta$ are invariant under $\mathbb{Z}^2$-translations, it follows from \eqref{sequencia-minimizante} that
	$$
	I_\beta(\overline{u}_n,\overline{v}_n) = I_\beta(u_n,v_n)=c_\beta+o_n(1) \quad \mbox{and} \quad J_\beta(\overline{u}_n,\overline{v}_n) = J_\beta(u_n,v_n)=0.
	$$
	The Lemma~\ref{limitacao-H1} yields $C_1>0$ such that
	\begin{equation}\label{limitacao-H1-C1}
		\|\overline{u}_n\|_{H^1(\mathbb{R}^2)}, \, \|\overline{v}_n\|_{H^1(\mathbb{R}^2)}\leq C_1, \quad \forall n \in \mathbb{N}.
	\end{equation}
	It remains to prove that there exists $C_2>0$ such that 
	$$
	\int_{\mathbb{R}^2 }\log(\lambda+|x|)\overline{u}_n^2, \, \int_{\mathbb{R}^2 }\log(\lambda+|x|)\overline{v}_n^2\leq C_2, \quad \forall n \in \mathbb{N}.
	$$
	Considering $x \in \mathbb{R}^2 \setminus B_{2R}(0)$ and $y \in B_R(0)$, notice that
	$$
	\lambda + |x-y| \geq \lambda+|x|-|y| \geq \lambda+|x|-R \geq \lambda+\frac{|x|}{2}  \geq \sqrt{\lambda+|x|}.
	$$
	Now, by using  \eqref{des-transladada-nao-nula}, we deduce that  
	$$
	\begin{aligned}
		C_3 &\geq\mathcal{V}_1(\overline{u}_n,\overline{v}_n)\\ 
		&\geq
		\frac{1}{2}\int_{\mathbb{R}^2 \setminus B_{2R}(0)}\int_{B_R(0)}\log(\lambda+|x|)\left(\overline{u}_n^2(y)+\overline{v}_n^2(y)\right)\left(\overline{u}_n^2(x)+\overline{v}_n^2(x)\right)\dy\dx
		\\
		&\geq \frac{\eta}{4}\int_{\mathbb{R}^2 \setminus B_{2R}(0)}\log(\lambda+|x|)\left(\overline{u}_n^2(x)+\overline{v}_n^2(x)\right)\dx.
	\end{aligned}
	$$
	Consequently,
	\begin{equation}\label{des-uniforme-lambda}
		\begin{aligned}
			\int_{\mathbb{R}^2 }\log(\lambda+|x|)\left(\overline{u}_n^2(y)+\overline{v}_n^2(y)\right)\dx &\leq \frac{4C_3}{\eta}
			+\log(\lambda+2R)\int_{B_{2R}(0)}\left(\overline{u}_n^2(y)+\overline{v}_n^2(y)\right)\dx
			\\
			& \leq C_4,
		\end{aligned}
	\end{equation}
	for some $C_4=C_4(\lambda)>0$, where we used \eqref{limitacao-H1-C1}. This step is completed.
	
	\bigskip
	
	\noindent\underline{Fourth step:} There exists $(u,v)\in\mathcal{M}_\beta$ such that
	$$
	I_\beta(u,v)=c_\beta, \quad \mbox{with} \ \ u,v\geq0.
	$$

	It follows from the third step that, up to a subsequence, $\overline{u}_n \rightharpoonup u$ and $\overline{v}_n \rightharpoonup v$ weakly in $W^\lambda$. We claim that $u\neq 0$ or $v\neq0$. If this is not true, $u=v=0$. Recall also that the compact embedding $W^\lambda\hookrightarrow L^s(\mathbb{R}^2)$ holds for each $2\leq s<\infty$, see for instance \cite{Cingolani-Weth}. Hence, {using the last assertion together with \eqref{desigualdadeV2}  and 	\eqref{psi} (with H\"older's inequality)}, we obtain
	$$
	\mathcal{V}_2(\overline{u}_n,\overline{v}_n)=\psi_\beta(\overline{u}_n,\overline{v}_n)=o_n(1).
	$$
	Similarly to \eqref{seq-min-conv-zero} one deduces 	
	$$
	\|\overline{u}_n\|_{H^1(\mathbb{R}^2)}+\|\overline{v}_n\|_{H^1(\mathbb{R}^2)}=o_n(1),
	$$
	which is a contradiction with Lemma~\ref{longe-do-zero}. Thus,  $u\neq 0$ or $v\neq0$.  
	
    In view of Lemma~\ref{projecao} there exists $t_0>0$ such that $(t_0^2u_{t_0},t_0^2v_{t_0}) \in \mathcal{M}_\beta$. We claim that $t_0 \in (0,1]$. In fact, analogously to the proof of \cite[Lemma 2.7]{cv2023} we infer that
	$$
	\mathcal{V}(u,v)\leq\liminf_{n\to \infty}\mathcal{V}(\overline{u}_n,\overline{v}_n).
	$$
	On other hand, from the compact embedding $W^\lambda\hookrightarrow L^s(\mathbb{R}^2)$  ($2\leq s<\infty$) together with \cite[Lemma  3.5]{Carvalho-Figueiredo-Maia-Medeiros}, one may deduce that
	$$
	\lim_{n\to \infty}\psi_\beta(\overline{u}_n,\overline{v}_n)= \psi_\beta(u,v)
	$$
	and
	$$
	\lim_{n\to \infty}\left(\|\overline{u}_n\|_{2}^2+ \|\overline{v}_n\|_{2}^2\right)^2= \left(\|u\|_{2}^2+ \|v\|_{2}^2\right)^2.
	$$
	The last three expressions combined with \eqref{Juv} and the fact that
	$$
	\|\nabla u\|_{2}^2+\|\nabla v\|_{2}^2 \leq\liminf_{n\to \infty}\left(\|\nabla \overline{u}_n\|_{2}^2+\|\nabla \overline{v}_n\|_{2}^2\right) 
	$$
	imply that
	\begin{equation}\label{condicao-t<1}
		J_\beta(u,v)\leq\liminf_{n\to \infty}J_\beta(\overline{u}_n,\overline{v}_n)=0.
	\end{equation}
	Therefore, $t_0 \in (0,1]$, see for instance Lemma~\ref{projecao}. 
	
	Finally, we can use \eqref{limitacao-inferior}
	to get
	\begin{equation}\label{est-nivel-tutv}
		\begin{aligned}
			c_\beta &\leq I_\beta(t_0^2u_{t_0},t_0^2v_{t_0}) \\
                &= I_\beta(t_0^2u_{t_0},t_0^2v_{t_0}) - \frac{1}{4}J_\beta(t_0^2u_{t_0},t_0^2v_{t_0})
			\\
			&=\frac{1}{8}\left(\|t_0^2u_{t_0}\|_{2}^2+ \|t_0^2v_{t_0}\|_{2}^2\right)^2 + \frac{2p-3}{4p}\psi_\beta(t_0^2u_{t_0},t_0^2v_{t_0}).
		\end{aligned}
	\end{equation}
	By \eqref{mud-var-norma-2p} (with $p=2$), it holds that
	$$
	\|t_0^2u_{t_0}\|_{2}^2 = t_0^2\|u\|_2^2 \quad \mbox{and} \quad \|t_0^2v_{t_0}\|_{2}^2 = t_0^2\|v\|_2^2.
	$$
	Moreover, the equality \eqref{l2p} yields
	$$
	\psi_\beta(t_0^2u_{t_0},t_0^2v_{t_0})=t_0^{4p-2}\psi_\beta(u,v).
	$$
	Since $t_0\in(0,1]$ and using the identities just above together with \eqref{est-nivel-tutv} and \eqref{limitacao-inferior}, we ensure that
	\begin{equation}\label{nivel-atingido}
		\begin{aligned}
			c_\beta &\leq I_\beta(t_0^2u_{t_0},t_0^2v_{t_0}) \\
                & \leq \frac{1}{8}\left(\|u\|_2^2+\|v\|_2^2\right)^2+\frac{2p-3}{4p}\psi_\beta(u,v) 
			\\
			&= \lim_{n\to \infty}\left[\frac{1}{8}\left(\|\overline{u}_n\|_2^2+\|\overline{v}_n\|_2^2\right)^2+\frac{2p-3}{4p}\psi_\beta(\overline{u}_n,\overline{v}_n)\right]\\
            & =\lim_{n\to \infty} I_\beta(\overline{u}_n,\overline{v}_n)=c_\beta.
		\end{aligned}
	\end{equation}
	Thus, $t_0=1$, because otherwise, $c_\beta<c_\beta$. Thereby,
	$$
	I_\beta(u,v)=c_\beta, \quad \mbox{with} \ \ (u,v) \in \mathcal{M}_\beta.
	$$
	Since
	$$
	c_\beta = I_\beta(u,v)=I_\beta(|u|,|v|) \quad \mbox{ and} \quad 0=J_\beta(u,v)=J_\beta(|u|,|v|),
	$$
	we can consider $u,v \geq0$ in $\mathbb{R}^2$. This ends the proof of this step.

	Now, we can use Proposition~\ref{esquema} to concludes that $(u,v)$ is a weak solution of \eqref{S}. The proof is completed.
\end{proof}

\subsection{Asymptotic behavior}
In this subsection we consider the behavior of solutions for system \eqref{S} when the parameter $\beta > 0$ goes to zero or infinity. Initially, we shall prove the following result:

\begin{lemma}
	For each $\lambda>e^{1/4}$, $u_\beta^\lambda\rightharpoonup u^\lambda$ and $v_\beta^\lambda\rightharpoonup v^\lambda$ weakly in $W^\lambda$, as $\beta\to0^+$, such that
	$$
	I_0(u^\lambda,v^\lambda)=c_0.
	$$
\end{lemma}

\begin{proof}
	We claim that $c_\beta \leq c_0$, for any $\beta \geq0$. In fact, consider $(u_0,v_0) \in \mathcal{M}_0 \subset  \left(W^\lambda\times W^\lambda\right)\setminus\{(0,0)\}$ such that
	$$
	I_0(u_0,v_0)=c_0.
	$$
	By Lemma~\ref{projecao}, there exists $t_\beta>0$ such that $(t_\beta^2(u_0)_{t_\beta},t_\beta^2(v_0)_{t_\beta}) \in \mathcal{M}_{\beta}$. Thus,
	$$
	\begin{aligned}
		c_\beta &\leq I_\beta(t_\beta^2(u_0)_{t_\beta},t_\beta^2(v_0)_{t_\beta}) \leq I_0(t_\beta^2(u_0)_{t_\beta},t_\beta^2(v_0)_{t_\beta}) \leq \max_{t>0}I_0(t^2(u_0)_{t},t^2(v_0)_{t}) 
		\\
		&= I_0(u_0,v_0)=c_0,
	\end{aligned}
	$$
	where we used \eqref{maximo}. So, the claim is verified.
	
	The last estimate combined with \eqref{estimativa-sequencia} and \eqref{limitacao-gradiente} imply that
	$$
	\|\nabla u_\beta^\lambda\|_{2}^2+\|\nabla v_\beta^\lambda\|_{2}^2+\left(\|u_\beta^\lambda\|_{2}^2+ \|v_\beta^\lambda\|_{2}^2\right)^2 \leq C_1,
	$$
	for some $C_1>0$ independent of $\beta$. Thus, the sequences $\{u_\beta^\lambda\}_{\beta \geq0}$ and $\{v_\beta^\lambda\}_{\beta \geq0}$ are bounded in $H^1(\mathbb{R}^2)$. Similarly to the proof of Proposition~\ref{atingibilidade}, we have that $u_\beta^\lambda\rightharpoonup u^\lambda$ and $v_\beta^\lambda\rightharpoonup v^\lambda$ weakly in $W^\lambda$, as $\beta\to0^+$, with $(u^\lambda,v^\lambda) \neq (0,0)$. By Lemma~\ref{projecao}, there exists $t_0>0$ such that $(t_0^2u^\lambda_{t_0},t_0^2v^\lambda_{t_0}) \in \mathcal{M}_0$. As in \eqref{condicao-t<1}, we can conclude that $J_0(u^\lambda,v^\lambda)\leq0$, and so $t_0 \in(0,1]$ (see Lemma~\ref{projecao}). 
	Combining \eqref{nivel-atingido} and using similar arguments and the fact that $c_\beta \leq c_0$, for any $\beta \geq0$, we obtain 
	$$
	c_0 \leq I_0(t_0^2u^\lambda_{t_0},t_0^2v^\lambda_{t_0}) \leq \lim_{\beta \to 0^+}I_\beta(u_\beta^\lambda,v_\beta^\lambda) = \lim_{\beta \to 0^+} c_\beta \leq c_0.
	$$
	Hence, $t_0=1$ because otherwise $c_0<c_0$. Therefore, $I_0(u^\lambda,v^\lambda)=c_0$, which finishes the proof.	
	\end{proof}
	
	\begin{lemma}
		For $\lambda>e^{1/4}$ fixed, we obtain that $\|u_\beta^\lambda\|_2 \to 0$ and $\|v_\beta^\lambda\|_2 \to 0$, as $\beta\to\infty$.	
	\end{lemma}
	
	\begin{proof}
		Let $\varphi \in C_0^\infty(\mathbb{R}^2) \setminus \{0\} \subset W^\lambda$. By Lemma~\ref{projecao}, for each $\beta \geq 0$ there exists $t_\beta>0$ such that $(t_\beta^2\varphi_{t_\beta},t_\beta^2\varphi_{t_\beta}) \in \mathcal{M}_{\beta}$. Thus,
		$$
		J_\beta(t_\beta^2\varphi_{t_\beta},t_\beta^2\varphi_{t_\beta})=0.
		$$
		We claim that
		\begin{equation}\label{convergencia-t-beta}
			\lim_{\beta\to \infty}t_\beta=0.
		\end{equation}
		{In fact, from \eqref{J-t2ut-t2vt} and \eqref{psi}, we have}
		$$
		4t_\beta^4\|\nabla \varphi\|_{2}^2-t_\beta^4\|\varphi\|_{2}^2 + 4t_\beta^4\log(t_\beta^{-1})\|\varphi\|_{2}^2+t_\beta^4\mathcal{V}(\varphi,\varphi)=\frac{(4p-2)}{2p}t_\beta^{4p-2}\left(2\|\varphi\|_{2p}^{2p}+2\beta\|\varphi\|_{2p}^{2p}\right).
		$$
		Hence,
		\begin{equation}\label{ali}
			t_\beta^{-4p+6}\left(\|\nabla \varphi\|_{2}^2-\|\varphi\|_{2}^2 + 4\log(t_\beta^{-1})\|\varphi\|_{2}^2+\mathcal{V}(\varphi,\varphi)\right) = \frac{(4p-2)}{2p}\left(2\|\varphi\|_{2p}^{2p}+2\beta\|\varphi\|_{2p}^{2p}\right).
		\end{equation}

		Suppose by contradiction that \eqref{convergencia-t-beta} is false. Then, up to a subsequence, we can assume that $\displaystyle\lim_{\beta\to \infty}t_\beta= \infty$ or $\displaystyle\lim_{\beta\to \infty}t_\beta= \xi>0$. Since $p>3/2$ we obtain that $-4p+6<0$. Moreover, $\displaystyle\lim_{\beta\to \infty}\log(t_\beta^{-1})/t_\beta^{4p-6}=0$. Thus, by setting $\displaystyle\lim_{\beta\to \infty}t_\beta= \infty$ in \eqref{ali} we obtain a contradiction. On the other hand, if $\displaystyle\lim_{\beta\to \infty}t_\beta= \xi>0$, then  
		$$
		\infty>\xi^{-4p+6}\left(\|\nabla \varphi\|_{2}^2-\|\varphi\|_{2}^2 + 4\log(\xi^{-1})\|\varphi\|_{2}^2+\mathcal{V}(\varphi,\varphi)\right) = \infty.
		$$
		Therefore the claim is proved.

    In what follows, we can use \eqref{funcional-utvt} together with \eqref{convergencia-t-beta} to obtain
$$
\lim_{\beta\to \infty}c_\beta\leq \lim_{\beta\to \infty} I_\beta(t_\beta^2\varphi_{t_\beta},t_\beta^2\varphi_{t_\beta})\leq \lim_{\beta\to \infty}\left[t_\beta^4\|\nabla \varphi\|_2^2+t_\beta^4\log(t_\beta^{-1})\|\varphi\|_2^2+\frac{t_\beta}{4}\mathcal{V}(\varphi,\varphi)\right]= 0.
$$
Finally, by using \eqref{estimativa-sequencia}, we see that
$$
0 \geq \lim_{\beta\to \infty}c_\beta \geq \frac{1}{8}\lim_{\beta\to \infty} \left(\|u_\beta^\lambda\|_2^2+\|v_\beta^\lambda\|_2^2\right)^2 \geq 0.
$$
This ends the proof.
\end{proof}

	\section{Regularity results and Pohozaev identity}\label{sec-regularidade-Pohozaev}

    In this section, we study the regularity of solutions for System \eqref{S} and a Pohozaev identity.

	\begin{lemma}\label{lema-regularidade-C2}
		Let $(u,v)$ be a weak solution of \eqref{S} obtained in Proposition~\ref{atingibilidade}. Then $u,v \in C^{2,\sigma}_{\mathrm{loc}}(\mathbb{R}^2)$, for some $\sigma\in(0,1)$. Furthermore, if $u \neq 0$, then $u>0$. Similarly, if $v\neq0$, we have $v>0$.
	\end{lemma}
	
	\begin{proof}
		Consider $R>0$. Since $u,v \in H^1(B_R(0))$, it follows that $u,v \in L^q(B_R(0))$, for any $q>1$. In light of \cite[Theorem 9.9]{Gilbarg-Trudinger}, we obtain $\phi_{u,v} \in W^{2,q}(B_R(0))$ and $\Delta \phi_{u,v}=2\pi\left(u^2+v^2\right)$ a.e. in $B_R(0)$. Now, noting that
		$$
			-\Delta u+\phi_{u,v} u=|u|^{2p-2}u+\beta|v|^p|u|^{p-2}u, \quad \mbox{in} \ \mathbb{R}^2,
		$$
		one has
		$$
		-\Delta u+u=h, \quad \mbox{in} \ B_R(0),
		$$
		where $h=|u|^{2p-2}u+\beta|v|^p|u|^{p-2}u-\phi_{u,v}+u$. Since $\phi_{u,v} \in W^{2,q}(B_R(0))$ and $u,v \in H^1(B_R(0))$, clearly $h \in L^q(B_R(0))$, for any $q>1$. By classical elliptic regularity theory, we deduce that $u \in W^{2,q}(B_R(0)) \hookrightarrow C^{0,\alpha_1}(B_R(0))$, for some $\alpha_1\in(0,1)$. Similarly, $v \in C^{0,\alpha_2}(B_R(0))$, for some $\alpha_2\in(0,1)$. Thus, we can conclude that $h \in C^{0,\alpha_3}(B_R(0))$, with $\alpha_3\in(0,1)$. Applying the regularity
		theorem of Agmon-Douglis-Nirenberg, $u \in C^{2,\sigma}(B_R(0))$. Similarly, $v \in C^{2,\sigma}(B_R(0))$, completing the first part of result. 
		
		Next, assume that $u \neq0$. Suppose by contradiction that there exists $x_0 \in \mathbb{R}^2$ such that $u(x_0)=0$. Since, $u \neq0$, there exists $x_1 \in \mathbb{R}^2$ such that $u(x_1)>0$. Let $\Omega$ be a ball centered in $x_0$ containing $x_1$ with radius $R>0$. Let $\widetilde{A}:\mathbb{R}^2 \to \mathbb{R}^2$ and $\widetilde{B}:\Omega \times \mathbb{R}^+ \to \mathbb{R}$ given by
		$$
		\widetilde{A}(\xi):= \xi \quad \mbox{and} \quad \widetilde{B}(x,z):=-\left(\int_{\mathbb{R}^2}\log(\lambda+|x-y|)\left(u^2(y)+v^2(y)\right)\dy\right)z.
		$$
		Since $(u,v)$ is a solution of \eqref{S}, with $u\geq0$, one deduces
		$$
		\mathrm{div}\widetilde{A}\left(\nabla u\right)+\widetilde{B}(x,u)=-|u|^{2p-2}u-\beta|v|^p|u|^{p-2}u-\zeta(x)u \leq 0, \quad \mbox{in} \ \Omega,
		$$
		where
		$$
		\zeta(x):=\int_{\mathbb{R}^2}\log(\lambda+|x-y|^{-1})\left(u^2(y)+v^2(y)\right)\dy.
		$$
		 Clearly $\widetilde{A} \in L^\infty_{\mathrm{loc}}(\mathbb{R}^2)$. Moreover, by \eqref{des-log-elementar} we obtain
		 \begin{equation}\label{conta-B}
		 	\left| \widetilde{B}(x,z) \right|  \leq \left[\log(\lambda+|x_0|+R)\left(\|u\|_2^2+\|v\|_2^2\right)+\left(\|u\|_{\lambda,*}^2+\|v\|_{\lambda,*}^2\right)\right]z=C_1z,
		 \end{equation}
		 and so $\widetilde{B} \in L^\infty(\Omega)$. 
		 
		 In order to apply the Strong Maximum
		 Principle \cite[Theorem 2.5.1]{Pucci}, we need to show that there exist constants $a_1,a_2>0$ and $b_1\geq 0$ such that, for any $x \in \Omega$, $0<z\leq 1$, and $|\xi|\leq 1$, 
		 $$
		 \left\langle \widetilde{A}(\xi),\xi\right\rangle \geq a_1|\xi|^2, \quad \left| \widetilde{A}(\xi) \right| \leq a_3|\xi|,  \quad \mbox{and} \quad \widetilde{B}(x,z) \geq -b_1z.
		 $$
		 The proof of the first two inequalities is immediate. 
		 The third estimate is a consequence of \eqref{conta-B}. Therefore, by Strong Maximum
		 Principle, either $u=0$ in $\Omega$ or $u>0$ in $\Omega$, which is a contradiction.
		\end{proof}

		 \begin{lemma}
		 	If $(u,v)$ is a weak solution of \eqref{S}, then \eqref{identidade-Pohozaev} holds.
		 \end{lemma}
		 
		 \begin{proof}
		 	Since $(u,v)$ is a weak solution of \eqref{S}, we have
		 	$$
		 	\left(-\Delta u+\phi_{u,v} u\right)\left(x \cdot \nabla u\right)=\left(|u|^{2p-2}u+\beta|v|^p|u|^{p-2}u\right)\left(x \cdot \nabla u\right), \quad \mbox{in} \ \mathbb{R}^2.
		 	$$
		 	A simple calculation shows that
		 	$$
		 	\Delta u \left(x \cdot \nabla u\right)= \mathrm{div}\left[\left(\left[x \cdot \nabla u\right]\nabla u\right)-x\left(\frac{|\nabla u|^2}{2}\right)\right],
		 	$$
		 	$$
		 	\phi_{u,v} u \left(x \cdot \nabla u\right)=\frac{1}{2}\mathrm{div}\left[ \phi_{u,v} u^2x\right]-\phi_{u,v}u^2-\frac{1}{2}u^2\left(x \cdot \nabla \phi_{u,v}\right),
		 	$$
		 	and
		 	$$
		 \mathfrak{f}(u)\left(x \cdot \nabla u\right)=\mathrm{div}\left[x \mathcal{F}(u)\right]-2\mathcal{F}(u),
		 	$$
		 	where $\mathcal{F}(s):=\int_{0}^{s} \mathfrak{f}(t)\,\mathrm{d}t$. The last expression implies that
		 	$$
		 	|u|^{2p-2}u\left(x \cdot \nabla u\right)=\mathrm{div}\left[x \frac{|u|^{2p}}{2p}\right]-\frac{|u|^{2p}}{p}
		 	$$ 
		 	and
		 	$$
		 	|v|^p|u|^{p-2}u\left(x \cdot \nabla u\right)=\mathrm{div}\left[x\frac{|v|^p|u|^p}{p}\right]-2\frac{|v|^p|u|^p}{p}.
		 	$$
		 	Using the Divergence theorem on a ball $B_R(0)$, we reach
		 	$$
		 	\begin{aligned}
		 		&-\int_{\partial B_R(0)}\left[\left(\left[x \cdot \nabla u\right]\nabla u\right)-x\left(\frac{|\nabla u|^2}{2}\right)\right]\cdot \nu\,\mathrm{d}\sigma+\frac{1}{2}\int_{\partial B_R(0)}\phi_{u,v} u^2(x\cdot \nu )\,\mathrm{d}\sigma \\
		 		&\quad-\int_{B_R(0)}\left(\phi_{u,v}u^2+\frac{1}{2}u^2\left(x \cdot \nabla \phi_{u,v}\right)\right)\dx
		 		\\
		 		&=\frac{1}{2p}\int_{\partial B_R(0)}|u|^{2p}(x \cdot \nu)\,\mathrm{d}\sigma-\frac{1}{p}\int_{B_R(0)}|u|^{2p}\dx+\frac{\beta}{p}\int_{\partial B_R(0)}|u|^p|v|^p(x\cdot \nu)\,\mathrm{d}\sigma
		 		\\
		 		&\quad-\frac{2\beta}{p}\int_{B_R(0)}|u|^p|v|^p\dx,
		 	\end{aligned}
		 	$$
		 	where $\nu=\nu(x)=x/R$ is the outward normal of $\partial B_R(0)$. By \cite[Lemma 2.1]{Severo-Gloss-Edcarlos}, if $\mathfrak{F} \in C(\mathbb{R}^2)\cap L^1(\mathbb{R}^2)$, then there exists a sequence $(R_n) \subset \mathbb{R}$, with $R_n \to \infty$ such that
		 	$$
		 	R_n \int_{\partial B_{R_n}(0)}|\mathfrak{F}(\sigma)|\,\mathrm{d}\sigma \to 0, \quad \mbox{as} \ \ n\to \infty.
		 	$$
		 	From this convergence and by the Cauchy-Schwarz inequality, we get
		 	$$
		 	\left| \int_{\partial B_{R_n}(0)}\left[\left(\left[x \cdot \nabla u\right]\nabla u\right)-x\left(\frac{|\nabla u|^2}{2}\right)\right]\cdot \nu\,\mathrm{d}\sigma \right| \leq \frac{R_n}{2}\int_{\partial B_{R_n}(0)}|\nabla u|^2\,\mathrm{d}\sigma \to 0, \quad \mbox{as} \  \ n \to \infty,
		 	$$
		 where we used $u \in C^2(\mathbb{R}^2)$ (see Lemma~\ref{lema-regularidade-C2}). By the regularity of $v$ and $\phi_{u,v}$ obtained in Lemma~\ref{lema-regularidade-C2}, we can conclude that the other integrals over $\partial B_{R_n}(0)$ converge to zero. Thus,
		 $$
		 \int_{\mathbb{R}^2}\left(\phi_{u,v}u^2+\frac{1}{2}u^2\left(x \cdot \nabla \phi_{u,v}\right)\right)\dx=\frac{1}{p}\int_{\mathbb{R}^2}|u|^{2p}\dx+\frac{2\beta}{p}\int_{\mathbb{R}^2}|u|^p|v|^p\dx.
		 $$
		 Noting that
		 $$
		 \begin{aligned}
		 	\frac{1}{2}\int_{\mathbb{R}^2}u^2 \left(x \cdot \nabla \phi_{u,v} \right)\dx&=\frac{1}{2}\int_{\mathbb{R}^2}\int_{\mathbb{R}^2}\frac{|x|^2-x\cdot y}{|x-y|^2}\left(u^2(y)+v^2(y)\right)u^2(x)\dy\dx
		 	\\
		 	&=\frac{1}{4}\int_{\mathbb{R}^2}\int_{\mathbb{R}^2}\frac{|x|^2-x\cdot y}{|x-y|^2}\left(u^2(y)+v^2(y)\right)u^2(x)\dy\dx 
		 	\\
		 	&\quad+  \frac{1}{4}\int_{\mathbb{R}^2}\int_{\mathbb{R}^2}\frac{|y|^2-y\cdot x}{|y-x|^2}\left(u^2(y)+v^2(y)\right)u^2(x)\dx\dy
		 	\\
		 	&=\frac{1}{4}\int_{\mathbb{R}^2}\int_{\mathbb{R}^2}\frac{|x|^2-2x\cdot y +|y|^2}{|x-y|^2}\left(u^2(y)+v^2(y)\right)u^2(x)\dy\dx \\
		 	&=\frac{1}{4}\left(\|u\|_2^2+\|v\|_2^2\right)\|u\|_2^2,
		 \end{aligned}
		 $$
		 we obtain
		 $$
		 \int_{\mathbb{R}^2}\int_{\mathbb{R}^2}\log(|x-y|)\left(u^2(y)+v^2(y)\right)u^2(x)\dy+\frac{1}{4}\left(\|u\|_2^2+\|v\|_2^2\right)\|u\|_2^2=\frac{1}{p}\|u\|_{2p}^{2p}+\frac{2\beta}{p}\int_{\mathbb{R}^2}|uv|^p\dx.
		 $$
		 Similarly,
		 $$
		 \int_{\mathbb{R}^2}\int_{\mathbb{R}^2}\log(|x-y|)\left(u^2(y)+v^2(y)\right)v^2(x)\dy+\frac{1}{4}\left(\|u\|_2^2+\|v\|_2^2\right)\|v\|_2^2=\frac{1}{p}\|v\|_{2p}^{2p}+\frac{2\beta}{p}\int_{\mathbb{R}^2}|uv|^p\dx.
		 $$
		 Therefore,
		 $$
		 \mathcal{V}(u,v)+\frac{1}{4}\left(\|u\|_2^2+\|v\|_2^2\right)^2=\frac{1}{p}\psi_\beta(u,v),
		 $$
		 completing the proof.
		\end{proof}

	\section{Proof of Theorem~\ref{vetorial-semitrivial}}\label{semitrivial-vetorial}

        This section is devoted to study the existence of vectorial and semi-trivial solutions.
	
	\subsection{Vectorial solution}
	\quad

	\begin{lemma}[Vectorial solution]\label{vetorial}
		For each $\beta>2^{p-1}-1$, the pair $(u,v)$ obtained in Proposition~\ref{atingibilidade} is a vector solution, that is, $u\neq0$ and $v\neq0$.
	\end{lemma}
	
	\begin{proof}
		Suppose by contradiction
		that $v=0$. Consider
		$$
		\widetilde{u}:=u \cos\theta \quad \mbox{and} \quad \widetilde{v}:=u\sin\theta.
		$$
		It is easy to see that 
		$$
		\|\nabla\widetilde{u}\|_2^2+\|\nabla\widetilde{v}\|_2^2 =\|\nabla u\|_2^2, \quad \left(\|\widetilde{u}\|_2^2+\|\widetilde{v}\|_2^2\right)^2 = \|u\|_2^4 \quad \mbox{and} \quad \mathcal{V}(\widetilde{u},\widetilde{v}) = \mathcal{V}(u,0).
		$$
		Since $(\widetilde{u},\widetilde{v})\neq(0,0)$, by Lemma~\ref{projecao}, there exists a unique $t_0>0$ such that $(t_0^2\widetilde{u}_{t_0},t_0^2\widetilde{v}_{t_0}) \in \mathcal{M}_\beta$. Thus, \eqref{funcional-utvt} combined with the last three identities imply that
		\begin{equation}\label{I-beta-u,0}
			c_\beta\leq I_\beta(t_0^2\widetilde{u}_{t_0},t_0^2\widetilde{v}_{t_0}) = \frac{t_0^4}{2}\|\nabla u\|_2^2+\frac{t_0^4\log(t_0^{-1})}{4}\|u\|_2^4+\frac{t_0^4}{4}\mathcal{V}(u,0)-\frac{t_0^{4p-2}}{2p}\psi_\beta(\widetilde{u},\widetilde{v}).
		\end{equation}
		
		Now, we estimate $\psi_\beta(\widetilde{u},\widetilde{v})$. {From \eqref{psi},}
		$$
		\psi_\beta(\widetilde{u},\widetilde{v})  = \left[\cos^{2p}\theta + \sin^{2p}\theta+2\beta\cos^p\theta\sin^p\theta\right]\|u\|_{2p}^{2p}.
		$$
		Choosing $\theta=\pi/4$, there holds
		$$
		\psi_\beta(\widetilde{u},\widetilde{v})  = (\beta+1)2^{1-p}\|u\|_{2p}^{2p}>\|u\|_{2p}^{2p}=\psi_\beta(u,0), \quad \mbox{if} \ \ \beta>2^{p-1}-1.
		$$
		Hence, the above inequality together with \eqref{I-beta-u,0}, \eqref{funcional-utvt}, and \eqref{maximo} imply that
		$$
		\begin{aligned}
			c_\beta&<\frac{t_0^4}{2}\|\nabla u\|_2^2+\frac{t_0^4\log(t_0^{-1})}{4}\|u\|_2^4+\frac{t_0^4}{4}\mathcal{V}(u,0)-\frac{t_0^{4p-2}}{2p}\psi_\beta(u,0)\\
            &=I_\beta(t_0^2u_{t_0},0)
			\\
			&\leq \max_{t>0}I_\beta(t^2u_{t},0)=I_\beta(u,0)=c_\beta,
		\end{aligned}
		$$
		which is a contradiction. Therefore, $u\neq 0$ and $v \neq0$. 
	\end{proof}
	
	\begin{lemma}
		Let $u$ be a solution of \eqref{P1}. Then, $\beta=2^{p-1}-1$, if and only if, $(u/\sqrt{2},u/\sqrt{2})$ is a solution of \eqref{S}.
	\end{lemma}
	
	\begin{proof}
		Firstly, assume that $\beta=2^{p-1}-1$. Since $u$ is a solution of \eqref{P1}, we can multiply by $1/\sqrt{2}$ to get
		$$
		-\Delta\left(\frac{u}{\sqrt{2}}\right)+\phi_u \left(\frac{u}{\sqrt{2}}\right) = \frac{1}{\sqrt{2}}|u|^{2p-2}u=2^{p-1}\left|\frac{u}{\sqrt{2}} \right|^{2p-2}\left(\frac{u}{\sqrt{2}}\right).
		$$
	Noting that $\phi_{u/\sqrt{2},u/\sqrt{2}}=\phi_{u}$, one deduces
	$$
	\begin{aligned}
		-\Delta\left(\frac{u}{\sqrt{2}}\right)+\phi_{u/\sqrt{2},u/\sqrt{2}}\left(\frac{u}{\sqrt{2}}\right) &= 2^{p-1}\left|\frac{u}{\sqrt{2}} \right|^{2p-2}\left(\frac{u}{\sqrt{2}}\right) \\
        & = (1+\beta)\left|\frac{u}{\sqrt{2}} \right|^{2p-2}\left(\frac{u}{\sqrt{2}}\right)
		\\
		&=\left|\frac{u}{\sqrt{2}} \right|^{2p-2}\left(\frac{u}{\sqrt{2}}\right)+\beta\left|\frac{u}{\sqrt{2}} \right|\left|\frac{u}{\sqrt{2}} \right|^{p-2}\left(\frac{u}{\sqrt{2}}\right).
	\end{aligned}
	$$
	Consequently, the pair $(u/\sqrt{2},u/\sqrt{2})$ is a solution of \eqref{S}.

	On the other hand, assume that $(u/\sqrt{2},u/\sqrt{2})$ is a solution of \eqref{S}. In view of $\phi_{u/\sqrt{2},u/\sqrt{2}}=\phi_{u}$ and the above identities, we infer that
	$$
	2^{p-1}\left|\frac{u}{\sqrt{2}} \right|^{2p-2}\left(\frac{u}{\sqrt{2}}\right)=(1+\beta)\left|\frac{u}{\sqrt{2}} \right|^{2p-2}\left(\frac{u}{\sqrt{2}}\right).
	$$
	Therefore, $2^{p-1}=\beta+1$, completing the proof.
	\end{proof}

	\subsection{Semi-trivial solution}
	\quad

	\begin{lemma}\label{lema-funcao-h}
		Let $p\geq2$ and $0\leq \beta<2^{p-1}-1$. Then the function defined by $h_\beta(s):=s^p+(1-s)^p+2\beta s^{p/2}(1-s)^{p/2}$, with $s\in[0,1]$, satisfies 
		$$
		h_\beta(s)<1, \quad \forall s \in(0,1).
		$$
	\end{lemma}
	
	\begin{proof}
	The proof follows by using \cite[Lemma 2.7]{Carvalho-Figueiredo-Maia-Medeiros}. On this subject we refer the reader also to \cite{hf1}. The details are omitted. 
	\end{proof}

		\begin{lemma}[Semi-trivial solution]
		For each $0\leq\beta<2^{p-1}-1$, the pair $(u_\beta,v_\beta)$ is a semi-trivial solution, that is, $u_\beta=0$ or $v_\beta=0$.
	\end{lemma}
	
	\begin{proof}
		Denoting $(u_\beta,v_\beta)$ by $(u,v)$ (this pair was obtained in Proposition~\ref{atingibilidade}),
		assume by contradiction that $u\neq0$ and $v\neq0$. By Lemma~\ref{lema-regularidade-C2},  we have that $u,v>0$. Using polar coordinates for the pair $(u,v)$, we can write
		$$
		(u,v)=(\rho\cos\theta,\rho\sin\theta),\quad\mbox{where}\quad \rho^2=u^2+v^2 \quad \mbox{and} \quad \theta=\theta(x)\in(0,\pi/2).
		$$ 
	It is straightforward to check that 
	$$
	\nabla u=\nabla\rho\cos\theta-\rho\nabla\theta\sin\theta \quad \mbox{and} \quad \nabla v=\nabla\rho\sin\theta+\rho\nabla\theta\cos\theta.
	$$
	Consequently,
	$$
	\begin{aligned}
		|\nabla u|^2+|\nabla v|^2&=\left(|\nabla \rho|^2\cos^2\theta-2\rho\cos\theta\sin\theta\nabla\rho\nabla\theta+\rho^2|\nabla\theta|^2\sin^2\theta\right)\\
		&\quad+\left(|\nabla\rho|^2\sin^2\theta+2\rho\sin\theta\cos\theta\nabla\rho\nabla\theta+\rho^2|\nabla\theta|^2\cos^2\theta\right)\\
		&=|\nabla\rho|^2+\rho^2|\nabla\theta|^2.
	\end{aligned}
	$$
	Thus,
	$$
	\|\nabla u\|_2^2+\|\nabla v\|_2^2 \geq  \|\nabla\rho\|_2^2, \quad \left(\|u\|_2^2+\|v\|_2^2\right)^2 = \|\rho\|_2^4 \quad \mbox{and} \quad \mathcal{V}(u,v)= \mathcal{V}(\rho,0).
	$$
	Since $(\rho,0) \neq(0,0)$, from Lemma~\ref{projecao}, there exists $t_0>0$ such that $(t_0^2\rho_{t_0},0) \in \mathcal{M}_\beta$. By \eqref{funcional-utvt},
	$$
	I_\beta(t_0^2u_{t_0},t_0^2v_{t_0}) \geq \frac{t_0^4}{2}\|\nabla\rho\|_2^2 + \frac{t_0^4\log(t_0^{-1})}{4}\|\rho\|_2^4+\frac{t_0^4}{4} \mathcal{V}(\rho,0)-\frac{t_0^{4p-2}}{2p}\psi_\beta(u,v).
	$$
	
	Now, we estimate $\psi_\beta(u,v)$. Since $\theta\in(0,\pi/2)$, then $0<\cos^2\theta<1$. Thus, we can apply Lemma~\ref{lema-funcao-h} to obtain
	$$
	\begin{aligned}
		\psi_\beta(u,v) &= \|u\|_{2p}^{2p}+\|v\|_{2p}^{2p}+2\beta\int_{\mathbb{R}^2}|uv|^p\dx
		\\
		&=\int_{\mathbb{R}^2}\left[(\cos^2\theta)^p+(\sin^2\theta)^p+2\beta(\cos^2\theta)^{p/2}(\sin^2\theta)^{p/2}\right]|\rho|^{2p}\dx
		\\
        &<\|\rho\|_{2p}^{2p}=\psi_\beta(\rho,0).
	\end{aligned}
	$$
	This inequality combined with the last estimate imply that
	$$
	I_\beta(t_0^2u_{t_0},t_0^2v_{t_0}) > I_\beta(t_0^2\rho_{t_0},0)\geq c_\beta,
	$$
	where we used \eqref{funcional-utvt}. Using \eqref{maximo}, since $(u,v)\in\mathcal{M}_\beta$, it holds that
	$$
	c_\beta < I_\beta(t_0^2u_{t_0},t_0^2v_{t_0}) \leq \max_{t>0}I_\beta(t^2u_{t},t^2v_{t}) = I_\beta(u,v)=c_\beta,
	$$
	which is a contradiction and this concludes the proof.
	\end{proof} 

    \section{\textbf{Declarations}}

\textbf{Ethical Approval }

It is not applicable.

\textbf{Competing interests}

There are no competing interests.

\textbf{Authors' contributions}

All authors wrote and reviewed this paper.

\textbf{Availability of data and materials}

All the data can be accessed from this article.

\end{document}